\documentclass[a4paper,10pt]{amsart}

\usepackage[utf8]{inputenc}
\usepackage{amssymb,amsthm,amsrefs,amsmath,epsfig,graphicx,color,transparent,enumerate}
\usepackage{hyperref}
\hypersetup{
    colorlinks,
    citecolor=black,
    filecolor=black,
    linkcolor=black,
    urlcolor=black
}

\DeclareMathOperator{\graph}{graph}
\DeclareMathOperator{\comp}{cc}
\DeclareMathOperator{\dist}{dist}
\DeclareMathOperator{\di}{d}
\DeclareMathOperator{\diam}{diam}
\DeclareMathOperator{\clos}{clos}

\newcommand{\liap}{\mathcal{L}}
\newcommand{\imax}[1]{\Gamma_{#1}}
\newcommand{\susp}[1]{{#1}'}

\newcommand{\N}{\mathbb{N}}
\newcommand{\Z}{\mathbb{Z}}
\newcommand{\R}{\mathbb{R}}

\newcommand{\K}{\mathcal{K}}

\newcommand{\KC}{\mathcal{C}}

\newcommand{\dder}{\ddot}
\newcommand{\vs}{\mathbb H}

\renewcommand{\epsilon}{\varepsilon}

\theoremstyle{plain}
\newtheorem{thm}{Theorem}[section]
\newtheorem{coro}[thm]{Corollary}

\newtheorem{prop}[thm]{Proposition}

\theoremstyle{definition}
\newtheorem{df}[thm]{Definition}

\theoremstyle{remark}
\newtheorem{rmk}[thm]{Remark}
\newtheorem{ejp}[thm]{Example}

\title{Isolated sets, catenary Lyapunov functions and expansive systems}

\author[A. Artigue]{Alfonso Artigue}
\address{Universidad de la Rep\'ublica, Departamento de Matem\'atica y Estad\'\i sitca del Litoral}
\email{aartigue@unorte.edu.uy}
\date{\today}

\begin{document}

\begin{abstract}
It is a paper about models for isolated sets and the construction of special hyperbolic Lyapunov functions. 
We prove that after a suitable surgery every isolated set is 
the intersection of an attractor and a repeller.
We give linear models for attractors and repellers. 
With these tools we construct hyperbolic Lyapunov functions and metrics around an isolated set
whose values along the orbits are catenary curves. 
Applications are given to expansive flows and homeomorphisms, 
obtaining, among other things, a hyperbolic metric on local cross sections for an arbitrary expansive flow on a compact metric space.
\end{abstract}
\maketitle

\section{Introduction}
A hanging chain describes a curve that is called \emph{catenary}.
Galileo's first approximation to this curve was a \emph{parabola} 
but, after the development of the infinitesimal calculus, this curve was shown to be related with hyperbolic cosines 
and it is not parabolic. 
As shown in \cite{DH} hyperbolic cosines also appear in the expression of the catenary, 
even if gravity is not assumed to be constant but associated with a varying potential $-1/r$, 
which is a more realistic model of gravity.

In the present paper we consider dynamical systems and the purpose is to construct 
Lyapunov functions whose values 
along the orbits 
have the harmony of a hanging chain.
They will be called \emph{catenary functions} and as we will see they are hyperbolic Lyapunov functions. 
We will show that every isolated set admits a catenary function defined on an isolating neighborhood.
The construction of these functions is based on two results.
First, in Theorem \ref{teoBigBang} we show 
that after a cut and paste procedure every isolated set is the intersection of an attractor with a repeller. 
Second, we prove in Theorem \ref{teoSemiConjLinAtt} that attractors and repellers have linear models.
Precise definitions and statements are given in the corresponding sections.
The applications to expansive systems, given in Sections \ref{secExpFlow} and \ref{secLyapHomeo}, 
are natural if an isolated set is found. 
As we will see, the difficulty of this task depends on the form of expansivity that we consider
and if we are dealing with flows or homeomorphisms.

Let us give an example illustrating the main concepts of the paper. 
Consider the differential equations in the plane: 
\begin{equation}
 \label{sis1}
 \left\{
 \begin{array}{l}
 \dot x=x,\\
 \dot y=-y. 
 \end{array}
 \right.
\end{equation}
This system determines a hyperbolic equilibrium point of saddle type at the origin. 
Its solutions are given by $\phi_t(x,y)=(xe^t,ye^{-t})$.
Consider the norm
$\liap(x,y)=|x|+|y|$.
We have that $\liap(\phi_t(x,y))=|x|e^t+|y|e^{-t}$ and 
$\dot \liap(x,y)=-|x|+|y|$. Consequently $\ddot \liap=\liap$. 
As usual, the dots indicate time derivatives.
When a function satisfies $\ddot \liap=\liap$ we call it \emph{catenary function} for the flow $\phi$. 
If $\di$ is the distance induced by the norm $\liap$, we have that
$\ddot \di=\di$, and we call it \emph{catenary metric}.
In this example $\Lambda=\{(0,0)\}$ is an isolated set because there is a 
compact neighborhood $N=[-1,1]\times[-1,1]$ of $\Lambda$, satisfying: 
the whole orbit of a point is contained in $N$ if and only if the point is in $\Lambda$.
In Theorem \ref{teoCatMetSet} we will show that 
every isolated point admits a catenary metric and that 
every isolated set admits a catenary pseudo-metric vanishing on pairs of points of the isolated set. 
This result will be proved for partial flows on metric spaces. 
In Theorem \ref{teoCatFunIso} we show that every isolated set admits a catenary function $\liap$.

The construction of Lyapunov functions is a classical tool for proving 
the asymptotic stability of an equilibrium point of a differential equation. 
In \cite{Massera49} Massera considered the converse problem 
in the setting of autonomous or periodic differential equations 
in $\R^n$.
He showed that every asymptotically stable singular point 
admits a positive and decreasing Lyapunov function of class $C^1$. 
From a topological viewpoint, i.e. Lyapunov functions of class $C^0$, simpler constructions can be made 
even on metric spaces, 
see for example \cites{AS,BaSz,Conley78,Hur,WY}.
In Section \ref{secAttRep} we will show that every attractor admits 
a positive and decreasing Lyapunov function $\liap$ satisfying $\dot\liap=-\liap$, 
which is a key step in the construction of a catenary function for an isolated set.
In Theorem \ref{teoCatAttDiff} we apply Massera's theorem to construct a differentiable Lyapunov function $\liap$ for an 
asymptotically stable equilibrium point of a differential equation in $\R^n$ satisfying 
$\dot\liap=-a\liap$ for a suitable positive constant $a$.

In topological dynamics the role of hyperbolicity can be played by expansivity.
Recall that a homeomorphism $f\colon X\to X$ of a compact metric space is \emph{expansive} 
if there is $\delta>0$ such that $\dist(f^n(x),f^n(y))<\delta$ for all $n\in\Z$ 
implies $x=y$. 
As noted by Utz in \cite{Utz} expansivity is related with isolated sets as follows: 
a homeomorphism is expansive if and only if the diagonal $\Lambda=\{(x,x):x\in X\}$ is 
an isolated set for the homeomorphism $(x,y)\mapsto (f(x),f(y))$ in $X\times X$.
In \cite{Utz} the expression \emph{isolated set} is not used, but in the proof of \cite{Utz}*{Theorem 2.1} 
the concept is clearly present.
For the study of expansive systems Lewowicz \cites{Lew80,Lew89} 
introduced Lyapunov functions, see also \cites{Ur,Vi,Pa}.  
He proved that expansiveness is equivalent with the existence of a function 
$\liap\colon N\subset X\times X\to\R$ defined on a neighborhood $N$ of the diagonal $\Lambda$ and 
satisfying that $\liap$ and $\dder \liap$ vanish on $\Lambda$ and are positive in $N\setminus \Lambda$.
In the discrete time case $\ddot \liap$ may be 
defined as $\dder \liap(x) = \liap(f(x))-2\liap(x)+\liap(f^{-1}(x))$.
We will show in Theorem \ref{catExp} that this function $\liap$ can be constructed 
in such a way that $\dder \liap=\liap$ also holds. 
Moreover, $\liap$ can be evaluated at every small compact subset of $X$ (not only at pairs of points).
In \cite{Fa} Fathi constructed an \emph{adapted hyperbolic metric} 
for an arbitrary expansive homeomorphism of a compact metric space $X$. 
It is a metric $\dist\colon X\times X\to\R$ defining the topology of $X$ for which there are $\delta>0$ and 
$\lambda>1$ such that if $\dist(x,y)<\delta$ then $\dist(f(x),f(y))\geq\lambda\dist(x,y)$ or 
$\dist(f^{-1}(x),f^{-1}(y))\geq\lambda\dist(x,y)$. 
In Theorem \ref{teoExpCatLocMet} we prove that every expansive homeomorphism admits a \emph{catenary local metric}. 
This is a metric $D_x$ defined on a neighborhood of each $x\in X$, varying continuously with $x$ 
and satisfying $\ddot D_x=D_x$.
In Section \ref{secCatMetExp} we study sufficient conditions in order to obtain 
a catenary metric, instead of a local metric, for an expansive homeomorphism.
For dynamical systems with continuous time we consider expansive flows as defined in \cite{BW}. 
In Section \ref{secExpFlow} we state this definition in terms of isolated sets. 
It is done using local cross sections. 
In Theorem \ref{teoCatMetExpFLow} we prove that every expansive flow admits a 
hyperbolic metric of catenary type defined on local cross sections. 

Let us explain the meaning of the \emph{catenary condition}. 
In the continuous-time case $\ddot \liap= \liap$ 
implies that $\liap(\phi_t(x))=ae^t+be^{-t}$ for suitable constants $a,b\in\R$ depending on $x$. 
As a consequence we obtain a function $\dot \liap^2-\liap^2$ that is a constant of motion.
In the discrete-time case, if for a fixed $x$ we define $u_n=\liap(f^n(x))$ we have 
that $u_n=\liap(f^n(x))=\dder \liap(f^n(x))=u_{n+1}-2u_n+u_{n-1}$ 
and $u_{n+1}-3u_n+u_{n-1}=0$. 
If $\lambda_s<1$ and $\lambda_u>1$ 
are the solutions of 
\begin{equation}
 \label{eqPolCat}
 \lambda^2-3\lambda +1=0
\end{equation}
then 
$u_n=a \lambda_s^n+b\lambda_u^n$.
This shows that the catenary property gives us a nice control of 
the hyperbolic behavior of the values that $\liap$ takes along the orbits of a discrete or continuous dynamical system.

Let us now describe the contents of the paper while explaining other results that we prove. 
In Section \ref{secIsolated} we consider isolated sets for partial flows. 
Partial flows appear naturally when the solutions of 
a differential equation are not defined for all $t\in\R$.
For a partial flow we consider its enveloping flow as defined in \cite{Ab}. 
The enveloping flow is an abstract continuation of 
the trajectories that are not defined for all $t\in\R$. 
In general this enveloping is defined in a topological space 
that may not be Hausdorff. 
This can be the case even if the original partial flow is defined on a metric space. 
In Example \ref{ejNoHausdorff} this phenomenon is illustrated.
Applying results from \cite{Ab} we solve the problem of finding Hausdorff enveloping spaces for isolated sets.
We show in Theorem \ref{teoEnvMet} that every isolated set has a neighborhood 
with metrizable enveloping space. 
This result allows us to understand that in the study of an isolated set there is no loss 
of generality if we assume that $\phi$ is a flow instead of a partial flow. 
This section also gives the correct setting for the study of expansive flows in Section \ref{secExpFlow} 
where expansivity is stated in terms of an isolated set of a partial flow that is not a flow. 
This is the reason why we start the paper studying isolated sets for partial flows.
But, the main result of Section \ref{secIsolated} is Theorem \ref{teoBigBang}. 
There, a special compactification of the enveloping flow is constructed that allows us to see 
the isolated set as a Morse set \cite{Conley78}, that is, the intersection of an attractor with a repeller. 
In this construction two fixed points, an attractor and a repeller, are used 
to compactify the space, obtaining something similar with a model of the physical universe 
starting with a Big Bang and ending in a Big Crunch. 

In Section \ref{secAttRep} we consider attractors and repellers. 
We prove that every attractor admits a  Lyapunov function satisfying $\dot \liap=-\liap$.
Also, a pseudo-metric $\di$ satisfying $\dot\di=-\di$ is constructed for an attractor. 
It is a metric if the attractor is a singleton. 
These results are based on the linear models obtained in Section \ref{secLInModel}.
It is well known that attractors admit positive and decreasing Lyapunov functions. 
In Theorem \ref{estAs}
we give a new proof of this result that is based on Whitney's size functions.

In Section \ref{secCatFunMet} we construct catenary functions for isolated sets. 
We prove, Proposition \ref{propCatHyp}, that catenary functions are hyperbolic Lyapunov functions in the sense of \cite{WY}.
In Theorem \ref{teoCatFun} we solve the equation $\ddot \liap=a\liap$ in an isolating neighborhood, 
where $a$ is a positive continuous function such that $\dot a=0$. 
The result is presented as a boundary value problem 
that gives a method to construct more Lyapunov functions of catenary type.
In Theorem \ref{teoCatMetSet}
we show that isolated points admits a catenary metric defined on an 
isolated neighborhood. For an arbitrary isolated set we obtain a pseudo-metric that 
vanishes on each pair of points in the isolated set. 
In Section \ref{subsecFakeSing} we study the structure of a flow near an isolated set.
We show in Theorem \ref{teoSemiConj}
that the dynamics in an isolating neighborhood of an isolated set
is semi-conjugate with a \emph{singular flow box}. 
A first approximation to this concept is as follows.
Let $v$ be a smooth vector field on a manifold $M$, take a non-equilibrium point $p\in M$ 
and a flow box $U$ containing $p$. 
Let $\rho\colon M\to \R$ be a non-negative smooth function vanishing only at $p$. 
For the flow induced by the vector field $\rho v$ we have that $U$ 
is a singular flow box. 
The equilibrium point created in this way is known as a \emph{fake singularity}. 
A generalization of this construction is consider on metric spaces.

The applications to expansive flows mentioned above are given in Section \ref{secExpFlow}. 
In Section \ref{secLyapHomeo} we consider discrete dynamical systems. 
Via suspensions we extend our results for isolated sets of homeomorphisms of metric spaces. 
More applications are given to expansive, cw-expansive homeomorphisms and other variations are considered.

\thanks{I thank Jos\'e Vieitez for useful conversations on Lyapunov functions and hyperbolic metrics of expansive homeomorphisms. 
I thank Dami\'an Ferraro for introducing me to the contents of \cite{Ab} related with enveloping spaces of partial actions.}

\section{An isolating universe} 
\label{secIsolated}

The purpose of this section is to prove that every isolated set can be seen as the intersection of 
an attractor and a repeller 
in what we call an \emph{isolating universe} for the isolated set. Such an intersection is called a \emph{Morse set} in \cite{Conley78}.
Let us give an example that illustrates what we will do. 
Consider the equations
\[
 \left\{
 \begin{array}{l}
  \dot x= \sin^2x+y^2\\ \dot y=0 
 \end{array}
 \right.
\]
in the cylinder $X=(\R/\pi)\times \R$. 
We have an equilibrium point at $(0,0)$ and an isolated set $\Lambda=\{(0,0)\}$. 
Consider $N=[-1,1]\times [-1,1]$. 
It is true that $N$ is an isolating neighborhood of $\Lambda$ 
but it is also true that every trajectory always returns to $N$. 
It will simplify many arguments and constructions to remove these recurrences.
When we restrict the dynamics to $N$ we obtain what is called a \emph{partial flow}. 
Since we are interested in the dynamics near the isolated set it is natural to consider partial flows instead of flows. 
Let us continue with the example. 
Once we have an isolating neighborhood as the rectangle $N$ 
we can abstractly continue the trajectories. 
This is the step 2 in Figure \ref{figSurgery}.
Now we compactify the space by adding two points. 
After this procedure we will see $\Lambda$ as the intersection of an attractor set
and a repeller set indicated with dotted lines in the final step of Figure \ref{figSurgery}. 
\begin{figure}[ht]
\center 
\includegraphics{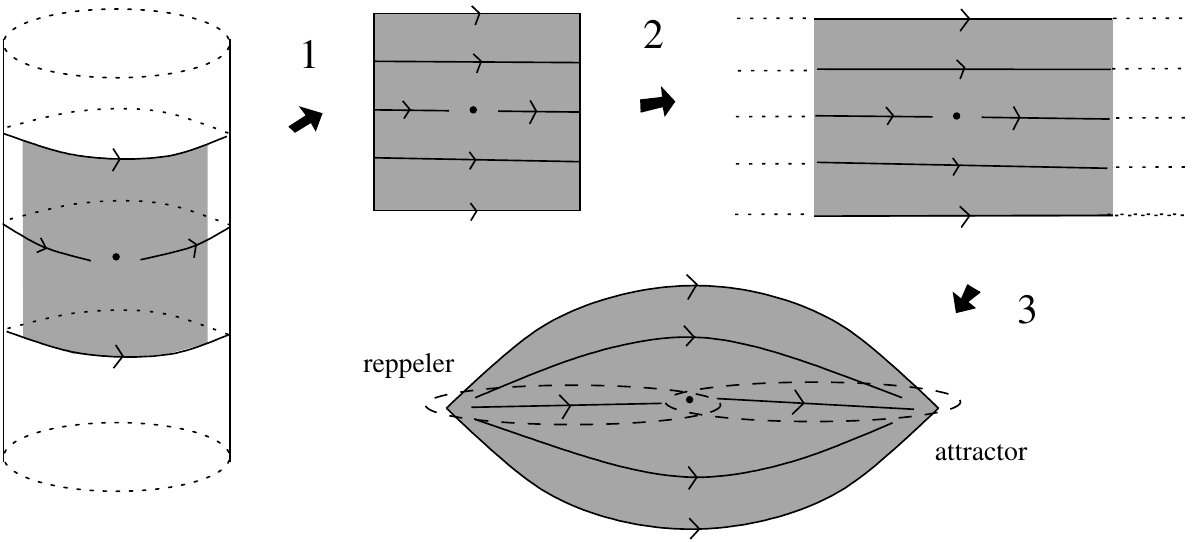}  
\caption{Surgery for the construction of an isolating universe. 
Starting with an isolated set the first step is to find a suitable isolating neighborhood, it can be an isolating block or a flow convex neighborhood. 
We obtain a partial flow.
Next, extend the trajectories without introducing recurrences, this is the enveloping flow. 
Finally add two singular points that compactify the space.}
\label{figSurgery}
\end{figure}

\subsection{Partial flows}
We start introducing partial flows and its basic properties.
Let $(X,\dist)$ be a metric space and consider an open set $\Gamma\subset \R\times X$.
\begin{df}A \emph{partial flow} on $X$ is a continuous function $\phi\colon\Gamma\to X$ 
such that: 
\begin{enumerate}
  \item for all $x\in X$ the set $\Gamma_{x}=\{t\in\R:(t,x)\in \Gamma\}$ is connected,
  \item $0\in \Gamma_{x}$ and $\Gamma_{\phi_t(x)}=\Gamma_{x}-t$ for all $(t,x)\in\Gamma$, 
  \item $\phi_0(x)=x$ for all $x\in X$ and 
$\phi_s(\phi_t(x))=\phi_{s+t}(x)$ whenever $s,t,s+t\in \Gamma_{x}$.
\end{enumerate}
If $\Gamma=X$ we say that $\phi$ is a \emph{flow}. 
\end{df}
In the context of differential equations $\Gamma_{x}$ is the \emph{maximal interval} of the solution through $x$. 

\subsubsection{Restricted flow}
\label{secRestFlow}
Let $\phi\colon\Gamma\to X$ be a partial flow on the metric space $X$. 
Consider an open set $U\subset X$ and 
define for $t\in\R$ the interval $J(t)=[0,t]$ if $t\geq 0$ and $J(t)=[t,0]$ for $t< 0$.
Consider the open set
$$\Gamma_U=\{(t,x)\in \R\times X: \phi_{J(t)}(x)\subset U\}$$
and define the partial flow $\psi=\phi|U\colon \Gamma_U\to X$ as 
$\psi_t(x)=\phi_t(x)$ if $x\in U$ 
and $\phi_{J(t)}(x)\subset U$.  
In this case we say that $\psi$ is the \emph{restriction} of $\phi$ on $U$.

\begin{rmk}
  In \cite{Conley78} there is a similar concept called \emph{local flow}. 
  It is essentially the restriction of a flow. 
  As we said, for the study of expansive flows in Section \ref{secExpFlow} 
   we need to start the theory from a partial flow..
\end{rmk}

\subsubsection{Morphisms of partial flows}
For $i=1,2$ let $\phi^i\colon\Gamma^i\to X_i$ be two partial flows.
A \emph{semi-conjugacy} is a continuous surjection 
$h\colon X_1\to X_2$ such that:
\begin{enumerate}
  \item $\Gamma^1_{x}=\Gamma^2_{h(x)}$ for all $x\in X_1$ and 
  \item $h(\phi^1_t(x))=\phi^2_t(h(x))$ for all $x\in X_1$ and for all $t\in \Gamma^1_{x}$.
\end{enumerate}
If in addition $h$ is a homeomorphism then $h$ is a \emph{conjugacy}.
For the partial flows $\phi^1,\phi^2$ as before define 
\[
  X_i^+=\{x\in X_i:\R^+\subset\Gamma^i_{x}\}
\]
and 
\[
  X_i^-=\{x\in X_i:\R^-\subset\Gamma^i_{x}\}
\]
for $i=1,2$.
\begin{prop}
  If $h\colon X_1\to X_2$ is a semi-conjugacy then 
  $h(X_1^+)=X_2^+$ and $h(X_1^-)=X_2^-$.
\end{prop}

\begin{proof}
If $x\in X^+_1$ then $\R^+\subset\Gamma^1_{x}$. 
Therefore, $\Gamma^1_{x}=\Gamma^2_{h(x)}$ because $h$ is a semi-conjugacy. 
Thus $\R^+\subset\Gamma^2_{h(x)}$ and $h(x)\in X^+_2$.
The rest of the proof is similar.
\end{proof}

\subsubsection{Extension of solutions}
Let $\phi\colon\Gamma\to X$ be a partial flow.
As in the theory of differential equations we can prove the following result.

\begin{prop}
\label{propExtSol}
If $\Gamma_{x}=(t_1,t_2)$ and $t_2$ is finite then $\phi_{s_n}(x)$ has no limit point for all $s_n\to t_2$. 
Similarly for $s_n\to t_1$ if $t_1$ is finite.
\end{prop}

\begin{proof}
  By contradiction assume that there is $s_n\in \imax{x}$ with $s_n\to t$ and $\phi_{s_n}(x)\to y$. 
  Since $\phi$ is defined on an open set $\Gamma$ we have that there are $\epsilon,\tau>0$ 
  such that if $\dist(z,y)<\epsilon$ then $(-\tau,\tau)\subset \imax{z}$. 
  Since $\phi_{s_n}(x)\to y$ we have that $\R^+\subset \imax{x}$, this is a consequence of (2) in the definition of partial flow. 
  This is a contradiction because we assumed that $t_2$ is finite. 
  The case of $t_1$ finite is analogous.
\end{proof}

\subsection{Isolated sets} 
Consider $\phi\colon\Gamma\to X$ a partial flow on the metric space $X$.
A subset
$\Lambda\subset X$ is $\phi$-\emph{invariant} 
if given $x\in \Lambda$ and $t\in\Gamma_x$ then $\phi_t(x)\in\Lambda$. 

\begin{rmk}
  If $\Lambda$ is $\phi$-invariant and compact then $\Gamma_x=\R$ for all $x\in\Lambda$. 
  It follows by Proposition \ref{propExtSol}.
\end{rmk}

\begin{df}
We say that $\Lambda$
is an \emph{isolated set} 
if there is a compact neighborhood $N$ of $\Lambda$ such that 
$\phi_{\imax{x}}(x)\subseteq N$ implies $x\in\Lambda$. 
In this case $N$ is an
\emph{isolating neighborhood} of $\Lambda$.
\end{df}
\begin{rmk}
  Every isolated set is compact and $\Gamma_x=\R$ for all $x\in \Lambda$.
\end{rmk}

\begin{prop}
If $N$ is an isolating neighborhood, $x\in N$, $\imax{x}=(t_1,t_2)$ and $t_1$ is finite then there is 
$t\in (t_1,0)$ such that $\phi_t(x)\notin N$. 
Analogously, if $t_2$ is finite then there is 
$t\in(0,t_2)$ such that $\phi_t(x)\notin N$. 
\end{prop}

\begin{proof}
It follows by Proposition \ref{propExtSol}.
\end{proof}

\subsection{Isolated points} 
\label{secIsoPOint} 
From our viewpoint, that is the construction of Lyapunov functions on an isolating neighborhood, 
it is not important the dynamics inside the isolated set $\Lambda$. 
Therefore, we will explain a standard procedure that collapses this set to a point. 
\begin{df} 
If $\{p\}$ is an isolated set we say that $p$ is an \emph{isolated point}.  
\end{df}
Given an isolated set $\Lambda$ of a partial flow $\phi$ consider an isolating neighborhood $N$ 
and the equivalence relation $\sim$ in $N$ generated by $x\sim y$ if $x,y\in\Lambda$.
Define $M=N/\sim$ with the quotient topology and $\pi\colon N\to M$ the projection. 
Since $\Lambda$ is invariant by $\phi$, a partial flow $\psi$ in $M$ is defined by $\psi_t(\pi(x))=\pi(\phi_t(x))$.

\begin{rmk}
  The projection $\pi$ is a semi-conjugacy between $\phi$ and $\psi$ and 
  $\Lambda$ is an isolated point of $\psi$.
\end{rmk}

\subsection{Flow convexity}
\label{dfAdNei}
An open set $U\subset X$ is $\phi$-\emph{convex} if 
$\phi_{[0,t]}(x)\subset\clos(U)$ with $x,\phi_t(x)\in U$ 
implies $\phi_{[0,t]}(x)\subset U$.
Given a set $A\subset X$ and $x\in A$ denote by $\comp_x(A)$ the connected component of $A$ that contains the point $x$.

\begin{prop}
\label{propConvex}
If $\Lambda$ is an isolated set and $N$ is an isolating neighborhood 
then there is a $\phi$-convex open set $U$ such that $\Lambda\subset U\subset N$.
\end{prop}

\begin{proof}
As we explained in the previous section, we do not lose generality assuming that 
$\Lambda=\{p\}$.
Let $r>0$ be such that $\clos(B_r(p))\subset N$. 
For $\rho\in (0,r)$ define the set
 \[
  U_\rho=\{x\in B_r(p):\comp_x(\phi_\R(x)\cap B_r(p))\cap B_\rho(p)\neq\emptyset\}.
 \]
By the continuity of $\phi$ we have that $U_\rho$ is an open set for all $\rho\in(0,r)$. 
Let us prove that if $\rho$ is sufficiently small then $U_\rho$ is $\phi$-convex. 
By contradiction, suppose that there are $\rho_n\to 0$, $a_n,b_n\in U_{\rho_n}$, $t_n\geq 0$ such that $b_n=\phi_{t_n}(a_n)$ and 
$l_n=\phi_{[0,t_n]}(a_n)\subset \clos(U_{\rho_n})$ but $l_n$ is not contained in $U_{\rho_n}$. 

If $l_n\subset B_r(p)$ then $l_n$ would be contained in $U_{\rho_n}$. 
Since we know that this is not the case
there is $s_n\in(0,t_n)$ such that $c_n=\phi_{s_n}(a_n)\in\partial B_r(p)$. 
Since $a_n, b_n\in U_{\rho_n}$ we know that 
$\comp_{a_n}(\phi_\R({a_n})\cap B_r(p))\cap B_\rho(p)\neq\emptyset$ and 
$\comp_{b_n}(\phi_\R({b_n})\cap B_r(p))\cap B_\rho(p)\neq\emptyset$.
Then, there must be $u_n<0$ and $v_n>0$ such that $\phi_{u_n}(c_n),\phi_{v_n}(c_n)\in B_{\rho_n}(p)$ 
with $\phi_{[u_n,v_n]}(c_n)\subset \clos (B_r(p)$, $u_n\to-\infty$ and $v_n\to+\infty$.
If $c$ is a limit point of $c_n$ we have that $\phi_\R(c)\subset B_r(p)$ and 
$c\neq p$. 
This contradicts that $\clos(B_r(p))$ is contained in an isolating neighborhood of $p$ and proves the result.
\end{proof}

\subsection{Enveloping flow}
\label{secEnvF}
Given a partial flow $\psi$ on a metric space $U$ 
consider the metric space $\R\times U$ 
and the flow $\psi'$ on $\R\times U$ 
given by $\psi'_t(s,x)=(s+t,x)$. 
Define an equivalence relation by $(r,x)\sim (s,y)$ if 
$y=\psi_{r-s}(x)$. 
The space $$U^e=\frac{\R\times U}{\sim}$$ is the \emph{enveloping space} 
and the induced flow $\psi^e$ on $U^e$ is the \emph{enveloping flow} 
of $\psi$. 
This construction is similar to the suspension flow of a homeomorphism and in \cite{Ab} 
it is considered for arbitrary partial actions of topological groups.

In $U^e$ we consider the quotient topology. 
It can be the case that the enveloping space is not Hausdorff 
even being that $U$ is a metric space as is our case. 
Let us give an example. 

\begin{ejp}
\label{ejNoHausdorff}
Let $U=\{(x,y)\in \R^2:x^2+y^2>0\}$ and 
consider the differential equations $\dot x=1, \dot y=0$.
In this case, 
the maximal interval of $(x,y)$ with $y\neq 0$ is $I_{(x,y)}=\R$. 
If $y=0$ we have that $I_{(x,0)}=(-x,+\infty)$ for $x>0$ 
and $I_{(x,0)}=(-\infty,x)$ for $x<0$.
In the enveloping flow two half lines are added in order to continue 
the positive trajectory of $(-1,0)$ and the negative trajectory of $(1,0)$. 
Notice that $\psi^e_2(-1,0)$ and $(1,0)$ are different points in $U^e$ 
and they do not have disjoint neighborhoods. 
Consequently $U^e$ is not a Hausdorff topological space. 
\end{ejp}
Consider the set
$
 \graph(\psi)=\{(t,x,y)\in \Gamma\times U: y=\psi_t(x)\}.
$
In \cite{Ab} it is shown that 
the enveloping space $U^e$ is Hausdorff if and only if $\graph(\psi)$ is a closed 
subset of $\R\times U\times U$. 
For the following result recall the restriction flow defined in Section \ref{secRestFlow}.
\begin{thm}
\label{teoEnvMet} 
If $\phi$ is a partial flow on $X$, $U\subset X$ is a 
 $\phi$-convex open set with compact closure
 and $\psi=\phi|U$ then 
 the enveloping space of $\psi$ is metrizable.
 \end{thm}

\begin{proof} 
Let us start showing that $\graph(\psi)$ is closed.
 Take sequences $t_n\to t\in\R$, $x_n\to x\in U$ and 
 $y_n\to y\in U$ such that $\psi_{t_n}(x_n)=y_n$. 
 In order to prove that $\graph(\psi)$ is closed 
 we will prove that $(t,x,y)\in\graph(\psi)$. 
 Without loss of generality assume that $t>0$.
 The continuity of $\phi$ implies that $\phi_t(x)=y$ 
 and $\phi_{[0,t]}(x)\subset\clos (U)$. 
 Since $x,y\in U$ and $U$ 
 is $\phi$-convex we have that 
 $\phi_{[0,t]}(x)\subset U$. 
 Then $y=\psi_t(x)$ and $\graph(\psi)$ is closed. 
 Applying \cite{Ab}*{Proposition 2.10} we have that 
 the enveloping space is Hausdorff.
 Since the closure of $U$ is compact we have that $U$ and $U^e$ are locally compact. 
 It is easy to see that $U^e$ has a countable base because, given a countable base $V_1,V_2,\dots$ of $U$, 
 the sets $\psi^e_q(V_i)$, with $q$ rational, form a countable base of $U^e$.
 Notice that $U$ has a countable base because it is metric and has compact closure.
 Finally, applying \cite{HY}*{Corollary 2-59} we conclude that $U^e$ is metrizable.
 \end{proof}

 \subsection{Isolating universe} 
Given an isolated set $\Lambda$ of a partial flow $\phi$ consider an
isolating neighborhood $N$. 

\begin{df}A flow $\psi$ on a compact metric space $Y$ is a an \emph{isolating universe}
of $\Lambda$ if: 
\begin{enumerate}
\item there is  
a homeomorphism $h\colon N\to M\subset Y$ conjugating $\phi|N$ with $\psi|M$, 
\item there are two singular points $\alpha,\omega\in Y$ 
such that $\alpha\neq\omega$, $\alpha,\omega\notin M$, 
\item for all $x\in Y$, $x\neq\alpha$, the positive orbit of $x$ 
converges to $h(\Lambda)$ or converges to $\omega$,
\item for all $x\in Y$, $x\neq\omega$, the negative orbit of $x$ 
converges to $h(\Lambda)$ or converges to $\alpha$.
\end{enumerate}
\end{df}
In this case we will identify $M$ with $N$ and $\Lambda$ with $h(\Lambda)$ and 
consider $N$ as a subset of $Y$. 
Define the sets 
\[
\begin{array}{cl}
  W^s(\Lambda)=\{x\in Y:\lim_{t\to+\infty}\dist(\psi_t(x),\Lambda)= 0\}, & \Lambda_\alpha=\{\alpha\}\cup\Lambda\cup W^s(\Lambda),\\
  W^u(\Lambda)=\{x\in Y:\lim_{t\to-\infty}\dist(\psi_t(x),\Lambda)= 0\}, & \Lambda_\omega=\{\omega\}\cup\Lambda\cup W^u(\Lambda).
\end{array}
\]

\begin{df}
  An isolated set $\Lambda$ is an \emph{attractor} if there is an isolating neighborhood 
$N$ such that if $x\in N$ and $x\notin \Lambda$ then $\phi_t(x)\notin N$ for some $t<0$.
We say that $\Lambda$ is a \emph{repeller} if when reversing time 
it is an attractor.
\end{df}

The following result allows us to see $\Lambda$ as the intersection of the 
attractor $\Lambda_\omega$ with the repeller $\Lambda_\alpha$.

\begin{thm}
\label{teoBigBang}
  Every isolated set admits an isolating universe $Y$ such that $\Lambda_\alpha$ is a repeller, $\Lambda_\omega$ is an attractor and 
  $\Lambda=\Lambda_\alpha\cap\Lambda_\omega$.
\end{thm}

\begin{proof}
Let $U$ be a $\phi$-convex neighborhood of $\Lambda$ given by Proposition \ref{propConvex}. 
Consider $\psi^e$ the enveloping flow on the enveloping space $U^e$ defined in Section \ref{secEnvF}. 
We known by Theorem \ref{teoEnvMet} that $U^e$ is a metrizable space. 
Define the set $Z=\{\alpha,\omega\}\cup U^e$ where $\alpha,\omega\notin U^e$ are different points.
Given $x\in U^e$ define 
$$I(x)=\{t\in\R:\psi^e_t\in U\}.$$ 
A basis of neighborhoods of $\omega$ is 
$$V_n(\omega)=\{x\in U^e: I(x)<-n\}\cup\{\omega\}$$
and 
a basis of neighborhoods of $\alpha$ is 
$$V_n(\alpha)=\{x\in U^e: I(x)>n\}\cup\{\alpha\}$$
for $n\geq 1$. This defines the topology of $Z$.
Let $N$ be a compact neighborhood of $\Lambda$ contained in $U$. 
Define $Y$ as the closure in $Z$ of $\psi_\R(N)$. 
In order to prove that $Y$ is a compact space let $\{U_a\}_{a\in A}$ be an arbitrary open covering 
of $Y$. A finite subcovering can be obtained as follows. 
Take $U_\alpha,U_\omega$ containing $\alpha$ and $\omega$. 
There is $t>0$ such that, $\psi_{[-t,t]}(N)$ contains $Y'=Y\setminus (U_\alpha\cup U_\omega)$. 
Since $N$ is compact we can take a finite covering of $Y'$. This proves that $Y$ is compact. 
It is easy to see that $Y$ is Hausdorff because the enveloping $U^e$ is metrizable. 
To show that $Y$ is metrizable it only rests to note that $Y$ has a countable basis and apply 
\cite{HY}*{Corollary 2-59}.
Therefore, $Y$ is a compact metric space. 

The flow $\psi$ can be extended to $Y$ by putting singular points at $\alpha$ and $\omega$, obtaining a continuous flow. 
Given a point $x\in N$ there are two possible cases for its positive orbit:
1) $\psi_{\R^+}(x)\subset N$, which implies that $\psi_t(x)\to\Lambda$ as $t\to+\infty$ and
2) $\psi_{\R^+}(x)\nsubseteq N$, in this case $\psi_t(x)\to\omega$.
Analogous for a negative orbit.
This proves that $Y$ is an isolating universe for $\Lambda$.
In Figure \ref{figBigBang} the construction is illustrated.

\begin{figure}[htb] 
\centering 
\def\svgwidth{320pt} 
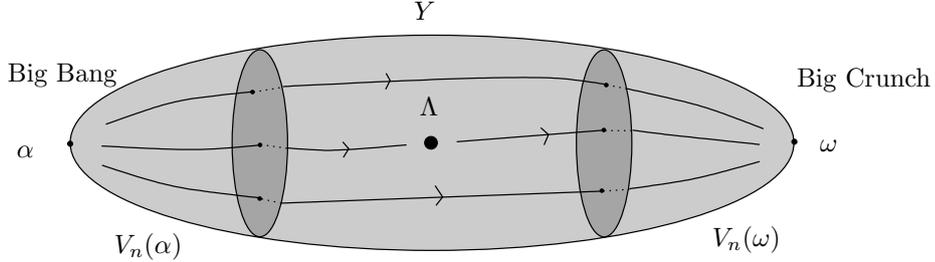 
\caption{Isolating universe for an isolated set $\Lambda$.}
\label{figBigBang}
\end{figure} 

We have that $\Lambda_\alpha$ is a repeller because
$N_\alpha=Y\setminus V_n(\omega)$ is an isolating neighborhood of $\Lambda_\alpha$ for all $n\geq 1$
and for all $x\in N_\alpha\setminus\Lambda_\alpha$ there is $t>0$ such that $\psi_t(x)\notin N_\alpha$.
Similarly we can see that $\Lambda_\omega$ is an attractor by considering $N_\omega=Y\setminus V_n(\alpha)$. 
Finally, $\Lambda=\Lambda_\alpha\cap\Lambda_\omega$ because $W^s(\Lambda)\cap W^u(\Lambda)=\Lambda$.
\end{proof}

This result implies that the construction of a Lyapunov function or a metric around an isolated set 
is reduced to the case of attractors and repellers.
 
\section{Attractors and repellers} 
\label{secAttRep}

In this section 
we construct linear models for attractors 
and we show that every attractor 
admits a positive Lyapunov function $\liap$ satisfying 
$\dot \liap=-\liap$, since $\liap$ is positive we have that it is decreasing. Similar results are concluded for repellers.

\subsection{Size functions} 
Given a compact set $N\subset X$
denote by $\K(N)$ the set of non-empty compact subsets of $N$.
In the set $\K(N)$ we consider the Hausdorff distance 
$\dist_H$ making $(\K(N),\dist_H)$ a compact metric space, see for example \cite{IN} for a proof.
Recall that 
\[
 \dist_H(A,B)=\inf\{\epsilon>0:A\subset B_\epsilon(B)\hbox{ and } B\subset B_\epsilon(A)\},
\]
where $B_\epsilon(C)=\cup_{x\in C} B_\epsilon(x)$ and $B_\epsilon(x)$ 
is the ball of radius $\epsilon$ and 
centered at $x$.
A \emph{size function} or a \emph{Whitney's function} is a continuous function $\mu\colon\K(X)\to\R$ satisfying:
\begin{enumerate}
 \item $\mu(A)\geq 0$ with equality if and only if $A$ is a singleton,
 \item if $A\subset B$ and $A\neq B$ then $\mu(A)<\mu(B)$. 
\end{enumerate}
A set $A$ is a \emph{singleton} if it only contains one point.
Let us recall how can be defined a size function. 
A variation of the construction given in \cite{Whitney33}, adapted for compact metric spaces, is the following. 
Let $q_1,q_2,q_3,\dots$ be a sequence dense in $N$. 
Define $\mu_i\colon \K(N)\to\R$ as 
\[
 \mu_i(A)=\max_{x\in A}\dist(q_i,x)-\min_{x\in A}\dist(q_i,x).
\]
The following formula defines a size function $\mu\colon \K(N)\to\R$
\[
 \mu(A)=\sum_{i=1}^\infty \frac{\mu_i(A)}{2^i},
\]
as proved in \cite{Whitney33}. 

\subsection{A decreasing Lyapunov function} 
A \emph{decreasing Lyapunov function} for an isolated set $\Lambda$ 
is a continuous function $\liap\colon U\to\R$ 
defined in a neighborhood of $\Lambda$ such that $\liap(\Lambda)=0$, 
and $\dot \liap$ is negative in $U\setminus\Lambda$.
We say that $\liap$ is \emph{positive} if $\liap(x)>0$ for all $x\in U\setminus\Lambda$.
As usual we define 
\[
  \dot \liap(x)=\lim_{t\to 0} \frac{\liap(\phi_t(x))-\liap(x)}t.
\]

\begin{thm}
\label{estAs}
Every attractor admits a positive and decreasing Lyapunov function.
\end{thm}

\begin{proof}
By the remarks in Section \ref{secIsoPOint} we do not lose generality if we assume that 
the attractor $\Lambda$ is a singleton $\Lambda=\{p\}$.
Then there are $\delta_0,\delta>0$ such that if $\dist(x,p)<\delta$ then 
 $\phi_t(x)\in B_{\delta_0}(p)$ for all $t\geq 0$ and $\phi_t(x)\to p$ as $t\to\infty$. 
 Define $U=B_\delta(p)$ and $\liap\colon U\to \R$ as
 \[
 \liap(x)=\mu(\{\phi_t(x):t\geq 0\}\cup\{p\})
 \]
 where $\mu$ is a size function.
 Since $\phi_t(x)\to p$ we have that 
 \begin{equation}\label{ecuO}
  O(x)=\{\phi_t(x):t\geq 0\}\cup\{p\}
 \end{equation}
 is a compact set for all $x\in U$. 
 Notice that if $t>0$ and $x\neq p$ then $O(\phi_t(x))\subset O(x)$ and the inclusion is proper.
 Therefore, $\liap(\phi_t(x))<\liap(x)$ because $\mu$ is a size function. 
 Also notice that $\liap(p)=0$ and $\liap(x)>0$ if $x\neq p$. 
 In order to prove the continuity of $\liap$, 
 we will prove the continuity of $O\colon U\to \K(X)$, the function defined by (\ref{ecuO}). 
 Since $\mu$ is continuous we will conclude the continuity of $\liap$. 
 
 Let us prove the continuity of $O$ at $x\in U$. 
 Take $\epsilon>0$. 
 By the asymptotic stability of $p$ there are $\rho,T>0$ such that 
 if $y\in B_\rho(x)$ then $\phi_t(y)\in B_{\epsilon/2}(p)$ for all $t\geq T$. 
 By the continuity of the flow, there is $r>0$ such that if 
 $y\in B_r(x)$ then $\dist(\phi_t(x),\phi_t(y))<\epsilon$ for all 
 $t\in[0,T]$. 
 Now it is easy to see that 
 if $y\in B_{\min\{\rho,r\}}(x)$ then 
 $\dist_H(O(x),O(y))<\epsilon$, proving the continuity of $O$ at $x$ and consequently the continuity of $\liap$.

 It can be the case that $\dot\liap$ does not exist and a well known procedure must be applied.
Define $\liap_1(x)=\int_0^\tau \liap(\phi_t(x))dt$ for a fixed 
$\tau>0$ small. 
In this way $\dot \liap_1$ exists and $\dot \liap_1(x)=\liap(\phi_\tau(x))-\liap(x)<0$ 
for all $x\in U\setminus \Lambda$. 
This proves the result.
\end{proof}

\subsection{Linear models}
\label{secLInModel}
Consider the normed vector space 
$\vs$ of real sequences $x\colon\N\to\R$, denoted by $x_n=x(n)$, such that 
\begin{equation}
  \label{ecuNorm}
  \|x\|^2=\sum_{i=1}^\infty x_i^2
\end{equation}
is convergent. 
Let $\psi\colon \R\times \vs\to \vs$ be the flow 
given by 
\begin{equation}
  \label{ecuLinFlow}
  \psi_t(x)=e^{-t}x.
\end{equation}
If $A\subset \vs$ is a compact set define $C_A$ as the \emph{cone} generated by $A$, that is, 
\[
  C_A=\{rx:r\in[0,1], x\in A\}.
\]
We say that $\psi$ restricted to $C_A$ has a \emph{linear attractor} at the origin $0_\vs$.

\begin{rmk}
\label{rmkEmbedding}
Let $\vs_1=\{x\in\vs:x_1=1\}$. 
By \cite{HY}*{Theorem 2-46} we know that every compact metric space is homeomorphic 
with a compact subset of $\vs_1$. 
\end{rmk}

\begin{df}
A \emph{cross section} for a flow $\phi$ on a metric space $X$ is a set $\Sigma\subset X$ such that $\phi_{(-t,t)}(\Sigma)$ is an open set 
for some $t>0$ and 
$\phi\colon (-t,t)\times\Sigma\to X$ is a homeomorphism onto its image.
\end{df}

\begin{thm}
\label{teoSemiConjLinAtt}  For every attractor $\Lambda$ of the flow $\phi$ 
  on the metric space $X$
  there are a neighborhood $N$ of $\Lambda$ 
  such that the restriction $\phi|N$
  is semi-conjugate with a linear attractor. 
  If $\Lambda=\{p\}$ we obtain a conjugacy.
\end{thm}

\begin{proof}
  Consider a decreasing Lyapunov function $\liap\colon U\to \R$ from Theorem \ref{estAs}, 
  where $U$ is an isolating neighborhood of $\Lambda$ with compact closure. 
  Denote by $m=\min \{\liap(x):x\in\partial U\}>0$. 
  Let $\Sigma=\liap^{-1}(m/2)$. It is a compact set and also a cross section because $\dot \liap<0$.
  Take from Remark \ref{rmkEmbedding} a homeomorphism 
  $i\colon \Sigma \to A\subset \vs_1$. 
  Define $$N=\Lambda\cup\{\phi_t(x):t\geq 0,x\in\Sigma\}$$ and 
  $h\colon N\to C_A$ by $h(\phi_t(x))=e^{-t}i(x)$ for $x\in\Sigma$ and $h(x)=0_\vs$ for $x\in\Lambda$. 
  By the definitions it is easy to see that $h$ is a semi-conjugacy.
  See Figure \ref{figLinMod}.
  It only rests to note that if $\Lambda=\{p\}$ then $h$ is injective and, consequently, a homeomorphism and a conjugacy.
\end{proof}

\begin{figure}[ht]
\center 
\includegraphics{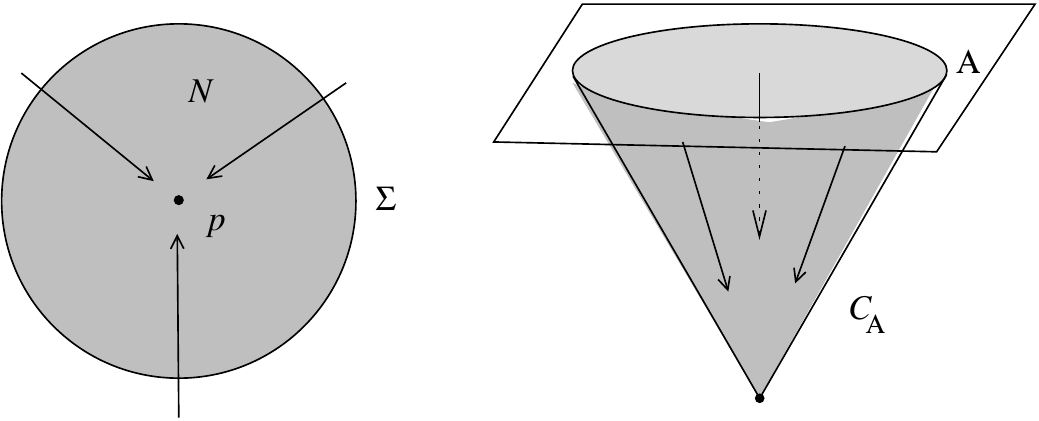}  
\caption{Linear model for an attractor.}
\label{figLinMod}
\end{figure}

\subsection{Special Lyapunov functions and metrics}
Recall that a \emph{pseudo-metric} on a set $N$ is a non-negative function $\di\colon N\times N\to\R$ such that 
\begin{enumerate} 
 \item $\di(x,x)=0$ for all $x\in N$, 
 \item $\di(x,y)=\di(y,x)$ for all $x,y\in N$ and 
 \item $\di(x,y)+\di(y,z)\geq\di(x,z)$ for all $x,y,z\in N$. 
\end{enumerate}
It is not required that $\di(x,y)=0$ implies $x=y$. 

\begin{prop}
\label{propCatAtt}
For every attractor $\Lambda$ there is a continuous pseudo-metric $\di\colon N\times N\to \R$, 
with $N$ a neighborhood of $\Lambda$, such that $\dot\di=-\di$ and 
$\di(x,y)=0$ if and only if $x=y$ or $x,y\in\Lambda$.
If in addition $\Lambda$ is a singleton then $\di$ is a metric.
\end{prop}

\begin{proof}
Consider from Theorem \ref{teoSemiConjLinAtt} a semi-conjugacy $h\colon N\to C_A$ 
between $\phi|N$ and the flow $\psi$ of equation (\ref{ecuLinFlow}) restricted to a cone $C_A$. 
Define $\di(x,y)=\|h(x)-h(y)\|$ where $\|\cdot\|$ is given in equation (\ref{ecuNorm}).
It is easy to prove that $\di$ is a pseudo-metric and $\di(x,y)=0$ if and only if $x=y$ or $x,y\in\Lambda$. 
Obviously, if $\Lambda$ is a singleton then $\di$ is a metric. 
Finally notice that 
$$\di(\phi_t(x),\phi_t(y))
=\|\psi_t(h(x))-\psi_t(h(y))\|
=e^{-t}\|h(x)-h(y)\|
=e^{-t}\di(x,y).$$
Therefore, $\dot\di=-\di$.
\end{proof}

Let $p$ be an asymptotically stable singular point, i.e., $\{p\}$ is an attractor. 
It is natural, if one looks for a decreasing Lyapunov function $\liap$ around $p$, 
to consider $\liap(x)=\dist(x,p)$. 
But one easily find examples, even hyperbolic linear systems in $\R^n$, 
for which $p$ can be asymptotically stable but $\dist(\phi_t(x),p)$ is not a decreasing function of $t$. 
From the previous result we obtain the following corollary that says that this idea works if the distance is changed.

\begin{coro}
If $p$ is asymptotically stable then there is a topologically equivalent metric $\di$ in a neighborhood of $p$ 
such that $\liap(x)=\di(x,p)$ is a decreasing Lyapunov function.
\end{coro}

\begin{proof}
 It is a direct consequence of Proposition \ref{propCatAtt}.
\end{proof}

\begin{prop}
\label{propLiapCatAtt}
 Every attractor $\Lambda$ admits a positive and decreasing Lyapunov function $\liap$
 satisfying $\dot \liap=-\liap$.
\end{prop}

\begin{proof}
Taking a linear model as in the proof of Proposition \ref{propCatAtt}, 
we have that $\liap(x)=\|h(x)\|$ satisfies $\dot \liap=-\liap$.
\end{proof}

\subsection{Repellers}
\label{secRep}
If $\Lambda$ is a repeller we obtain results that are similar with those that we proved for attractors. 
Let us remark that a pseudo-metric $\di$ as in Proposition \ref{propCatAtt} 
will satisfy $\dot\di=\di$ for a repeller $\Lambda$. 
Also, the Lyapunov function $\liap$ for a repeller, analogous to Proposition \ref{propLiapCatAtt}, 
satisfies $\dot\liap=\liap$.

\section{Catenary functions} 
\label{secCatFunMet}
In this section we will construct catenary functions, metrics and pseudo-metrics for an isolated set. 
In Section \ref{secCatAtt} we construct a differentiable catenary function for an 
asymptotically stable equilibrium point of a differential equation in $\R^n$.

\subsection{Catenary functions}
\label{secCatIso}
Let $\Lambda$ be an isolated set of the partial flow $\phi\colon\Gamma\to X$ of the metric space $X$.

\begin{df}
\label{dfCatPsMet}
A \emph{catenary pseudo-metric} for the isolated set $\Lambda$ 
is a continuous pseudo-metric $\di\colon N\times N\to\R$ 
defined on an isolating neighborhood $N$
such that: 
\begin{enumerate}
  \item $\ddot\di=\di$ and 
  \item $\di(x,y)=0$ if and only if $x=y$ or $x,y\in\Lambda$.
\end{enumerate}
If $\di(x,y)=0$ implies $x=y$ (i.e. $\Lambda$ is a singleton) we say that $\di$ is a \emph{catenary metric}.
\end{df}

\begin{thm}
\label{teoCatMetSet}
Every isolated set admits a catenary pseudo-metric. 
If the isolated set is a singleton then we obtain a catenary metric.
\end{thm}

\begin{proof}
Let $\Lambda$ be an isolated set with isolating neighborhood $N$. 
By Theorem \ref{teoBigBang} we can assume that $N\subset Y$ with $Y$ an isolating universe for $\Lambda$ 
and $\Lambda$ is the intersection of the attractor $\Lambda_\omega$ and the repeller $\Lambda_\alpha$.
By Proposition \ref{propCatAtt} we know that there is a 
continuous pseudo-metric $\di_\omega$ on an isolating neighborhood $N_\omega$ of $\Lambda_\omega$ 
satisfying $\dot\di_\omega=-\di_\omega$ 
and $\di(x,y)=0$ if and only if $x=y$ or $x,y\in\Lambda_\omega$. 
In addition we can assume that $N\subset N_\omega$. 
Analogously, by the remarks on Section \ref{secRep}, 
we have a pseudo-metric $\di_\alpha$ on an isolating neighborhood $N_\alpha$ 
satisfying $\dot\di_\alpha=\di_\alpha$ and 
$\di(x,y)=0$ if and only if $x=y$ or $x,y\in\Lambda_\alpha$. 
We will suppose that $N\subset N_\alpha$. 
We have that $\di=\di_\alpha+\di_\omega$ is a 
continuous pseudo-metric on $N_\alpha\cap N_\omega$.
Since $\dot\di_\alpha=\di_\alpha$ and $\dot\di_\omega=-\di_\omega$ we have that 
$\ddot\di=\di$. 
If $\di(x,y)=0$ then $\di_\alpha(x,y)=\di_\omega=0$. 
If $x\neq y$ then $x,y\in\Lambda_\alpha$ and $x,y\in\Lambda_\omega$. 
Then $x,y\in\Lambda$. This proves that $\di$ is a catenary pseudo-metric for $\Lambda$ 
defined in $N_\alpha\cap N_\omega$ that contains the isolating neighborhood $N$ of $\Lambda$.
If $\Lambda$ is a singleton then $\di$ is a metric and consequently a catenary metric.
\end{proof}

\begin{df}
A \emph{catenary Lyapunov function} or \emph{catenary function} 
for an isolated set $\Lambda$ is 
a continuous function $\liap\colon N\to\R$, 
with $N$ an isolating neighborhood of $\Lambda$,
satisfying $\liap(\Lambda)=0$, $\liap(x)>0$ for all $x\in N\setminus\Lambda$
and $\ddot \liap=\liap$ for all $x\in N$
\end{df}

\begin{thm}
\label{teoCatFunIso}
Every isolated set admits a catenary function.
\end{thm}

\begin{proof}
As in the proof of Theorem \ref{teoCatMetSet} we assume that $\Lambda$ is embedded in an isolating universe $Y$ 
and $\Lambda=\Lambda_\alpha\cap\Lambda_\omega$.
Since $\Lambda_\omega$ is an attractor, 
by Proposition \ref{propLiapCatAtt} there is positive and decreasing Lyapunov function $\liap_\omega$ 
on a neighborhood $N_\omega$ of $\Lambda_\omega$ such that $\dot\liap=-\liap$. 
Also, by the remarks on Section \ref{secRep}, 
since $\Lambda_\alpha$ is a repeller we have a positive and increasing Lyapunov function 
$\liap_\alpha$ such that $\dot\liap_\alpha=\liap_\alpha$. 
Then we have that 
$\liap=\liap_\alpha+\liap_\omega$
is a catenary function for $\Lambda$. 
\end{proof}

\begin{rmk}
  Another construction for the proof of Theorem \ref{teoCatFunIso} is as follows. 
  Consider a catenary pseudo-metric $\di$ for $\Lambda$ given by Theorem \ref{teoCatMetSet} 
  and define $\liap(x)=\di(x,\Lambda)=\inf_{y\in\Lambda}\di(x,y)$. 
  Note that in this case $\inf_{y\in\Lambda}\di(x,y)=\di(x,z)$ for all $z\in\Lambda$.
\end{rmk}

For the study of isolated sets, Conley \cite{Conley78} considered decreasing Lyapunov functions vanishing on $\Lambda$. 
In general, such a function will not have a definite sign.
A special Lyapunov function of this type can be constructed as follows.

\begin{coro}
\label{coroDecLyapIsoSet} Given an isolated set $\Lambda$ there is a continuous 
function $\liap_1\colon N\to\R$, where $N$ is an isolating neighborhood 
 of $\Lambda$, such that $\liap_1(\Lambda)=0$, $\dot\liap_1(x)<0$ for all $x\in N\setminus\Lambda$ 
 and $\ddot\liap_1=\liap_1$ in $N$.
\end{coro}

\begin{proof}
Consider a catenary function $\liap$ given by Theorem \ref{teoCatFunIso}.
The result follows by considering $\liap_1=-\dot\liap$.
\end{proof}

\subsection{Catenary functions are hyperbolic}
In \cite{WY} isolated sets for smooth vector fields on manifolds are considered. 
They study, among other things, the relationship between isolating blocks and hyperbolic Lyapunov functions.
The following is our topological version of \cite{WY}*{Definition 1.5}.

\begin{df}
A \emph{hyperbolic Lyapunov function} 
for an isolated set $\Lambda$ is 
a continuous function $\liap\colon N\to\R$, 
with $N$ an isolating neighborhood of $\Lambda$,
satisfying: 
\begin{enumerate}
  \item $\liap(\Lambda)=0$, $\liap(x)>0$ for all $x\in N\setminus\Lambda$ and 
  \item if $\dot\liap(x)=0$ and $x\in N\setminus\Lambda$ then $\ddot\liap(x)\neq 0$.
\end{enumerate}
\end{df}

\begin{prop}
\label{propCatHyp}
  Every catenary function is a hyperbolic Lyapunov function.
\end{prop}

\begin{proof}
Since a catenary function $\liap$ is positive in $N\setminus\Lambda$ and $\ddot \liap=\liap$ we have 
that $\ddot\liap$ is positive in $N\setminus\Lambda$. Then $\liap$ is a hyperbolic Lyapunov function.
\end{proof}

\subsection{More catenary functions}
\newcommand{\liapa}{\liap_1}
\newcommand{\liapb}{\liap_2}
\newcommand{\block}{\mathcal B}
Let $N$ be an isolating neighborhood of the isolated set $\Lambda$. 
Let $\liapa$ be a catenary function defined on $N$. 
Fix $\delta>0$ such that the set
$$\block=\{x\in N:\liapa (x)\leq\delta\}$$
is contained in the interior of $N$. 

\begin{rmk}
  A set like $\block$ is sometimes called \emph{isolating block}. 
  Since a precise definition of isolating block is a little involved and depends on the author we will not use this terminology. 
  See for example \cites{Conley78,Ch} for more on this. 
\end{rmk}

The definitions that follows are standard.
Consider the sets 
\[
  \Sigma_s=\{x\in\partial \block:\exists t<0\hbox{ with } \phi_{(t,0)}(x)\cap \block=\emptyset\},
\]
\[
  \Sigma_u=\{x\in\partial \block:\exists t>0\hbox{ with } \phi_{(0,t)}(x)\cap \block=\emptyset\}.
\]
\begin{rmk}
  With the previous notation we have that $\partial B=\Sigma_s\cup \Sigma_u$. 
  This follows because $\ddot \liapa>0$ in $\partial \block$. 
  In fact, $\Sigma_s=\{x\in\partial\block:\dot\liapa(x)\leq 0\}$ and 
  $\Sigma_u=\{x\in\partial\block:\dot\liapa(x)\geq 0\}$. 
  Moreover, $\Sigma_s\setminus \Sigma_u$ and $\Sigma_u\setminus\Sigma_s$ are cross sections.
\end{rmk}
Define
$$W^s=\{x\in \block:\phi_{\R^+}(x)\subset \block\},$$
$$W^u=\{x\in \block:\phi_{\R^-}(x)\subset \block\}.$$
Denote by $[-\infty,+\infty]$ the compactification with two points of $\R$.
Define the functions 
$T^s,T^u\colon \block\to [-\infty,+\infty]$ 
by
\[
  T^s(x)=\inf\{t\leq 0:\phi_{[t,0]}(x)\subset \block\},
\]
\[
  T^u(x)=\sup\{t\geq 0:\phi_{[0,t]}(x)\subset \block\}.
\]
It is easy to see that $T^s$ and $T^u$ are continuous.
Define $T\colon \block\to [0,+\infty]$ by $$T(x)=T^u(x)-T^s(x).$$ 
Since $T^s$ and $T^u$ are continuous we have that $T$ is continuous.
Notice that $T(x)=+\infty$ if and only if $x\in W^s\cup W^u$.
Introduce the notation 
\[
\left\{
\begin{array}{l}
\pi_sx=\phi_{T^s(x)}(x)\hbox{ for }x\notin W^u,\\
\pi_ux=\phi_{T^u(x)}(x)\hbox{ for }x\notin W^s.
\end{array}
\right.
\]

\begin{rmk}
  Take $x\in \block\setminus(W^s\cup W^u)$. 
  In this case $\pi_sx,\pi_ux\in\partial \block$. 
  Assume that $\ddot\liap=a^2\liap$ on the orbit of $x$ for $a\in\R$ and some function $\liap$ defined on $\block$. 
  Then, there are constants $b,c\in\R$ such that $\liap(\phi_t(x))=be^{at}+ce^{-at}$ 
  if $T^s(x)\leq t\leq T^u(x)$. 
  If we know the values $\liap(\pi_sx)$ and $\liap(\pi_ux)$ we can calculate the values of $b$ and $c$. 
  This is what we did in order to obtain the expression of $\liap$ given at equation (\ref{ecuCatFun}) in the next proof.
\end{rmk}

\begin{thm}
\label{teoCatFun}
Given a continuous function $a\colon \block\setminus \Lambda\to \R^+$ with $\dot a=0<\inf a$ and 
a continuous function $f\colon \partial \block\to\R$ 
then there is a unique continuous function $\liapb\colon \block\to\R$ such that 
\begin{enumerate}
\item $\liapb(\Lambda)=0$,
\item $\ddot \liapb(x)=a^2(x)\liapb(x)$ for all $x\in \block$,
\item $\liapb(x)=f(x)$ for all $x\in \partial \block$.
\end{enumerate}
If in addition $f>0$  
then $\liapb(x)>0$ for all $x\in \block\setminus\Lambda$.
\end{thm}

\begin{proof}
Without loss of generality assume that $\Lambda=\{p\}$.
Take $x$ in the interior of $\block$ (where $T(x)>0$) and define
\begin{equation}
  \label{ecuCatFun}
   \liapb(x)=\frac{f(\pi_sx)\sinh(a(x)T^u(x))-f(\pi_ux)\sinh(a(x)T^s(x))}{\sinh(a(x)T(x))}
\end{equation}
if $x\notin W^s\cup W^u$. 
For $x\in W^s$, $x\neq p$,  define 
\[
 \liapb(x)=f(\pi_sx)e^{a(x)T^s(x)}
\]
and if $x\in W^u$, $x\neq p$, define 
\[
 \liapb(x)=f(\pi_ux)e^{-a(x)T^u(x)}.
\]
Finally, $\liapb(p)=0$ and $\liapb(x)=f(x)$ for all $x\in\partial\block$.
Given the differential equation and the boundary condition there is no other possible choice for $\liapb$.
We have to check that it works.
Notice that $T^s(\phi_t(x))=T^s(x)-t$ and 
$T^u(\phi_t(x))=T^u(x)-t$. 
Then $\dot T^s(x)=\dot T^u(x)=-1$ and $\dot T(x)=0$ for all $x\in \block$.
Therefore $\liapb$ satisfies $\ddot \liapb(x)=a^2(x) \liapb(x)$ for all 
$x\in \block$. 

Let us prove the continuity of $\liapb$.
We show the continuity at the point $p$. Consider a sequence $x_n\to p$. 
First suppose that $x_n\in W^s$. 
In this case $\liapb(x_n)=f(\pi_ux_n)e^{-a(x_n)T^u(x_n)}$.
Since $T^u(x_n)\to+\infty$, $f$ is bounded and $\inf a>0$ we have that $\liapb(x_n)\to 0=\liapb(p)$.
For $x_n\in W^u$ a similar argument proves that $\liapb(x_n)\to 0$. 
Consider $x_n\notin W^s\cup W^u$. 
In this case $T^u(x_n),T(x_n),-T^s(x_n)\to+\infty$. 
Therefore, we have the following equivalent expressions (where $x$ denotes $x_n$)
\[
\begin{array}{ll}
 \displaystyle\frac{\sinh(a(x)T^u(x))}{\sinh(a(x)T(x))}&=\displaystyle\frac{e^{a(x)T^u(x)}-e^{-a(x)T^u(x)}}{e^{a(x)T(x)}-e^{-a(x)T(x)}}\\
 & \sim \displaystyle\frac{e^{a(x)T^u(x)}}{e^{a(x)[T^u(x)-T^s(x)]}}
 = e^{a(x)T^s(x)}
\end{array}
\]
and $e^{a(x_n)T^s(x_n)}\to 0$ because $\inf a>0$. 
Since $f$ is bounded we have that 
$$\frac{f(\pi_sx)\sinh(a(x)T^u(x))}{\sinh(a(x)T(x))}\to 0.$$
In the same way we can prove that
\[
 f(\pi_u(x))\frac{\sinh(a(x)T^s(x))}{\sinh(a(x)T(x))}\to 0
\]
if $x\to p$. Then, $\liapb(x_n)\to 0$ if $x_n\to p$. This proves the continuity 
of $\liapb$ at $p$. 
The proof of the continuity at other points is similar. 

Now suppose that $f$ is positive. 
We have to prove that $\liapb(x)>0$ for all $x\neq p$. 
We know that $u(t)=\liapb(\phi_t(x))=Ae^{a(x)t}+B e^{-a(x)t}$ 
and $u(t_0),u(t_1)>0$ if $\phi_{t_0}(x)\in\Sigma_s$ and 
$\phi_{t_1}(x)\in\Sigma_u$. 
We have to prove that $u(t)>0$ if $t_0<t<t_1$. 
If $A=0$ or $\block=0$ it is trivial. 
If $A>0$ and $\block>0$ it is also trivial. 
If $A$ and $\block$ have different signs then $\dot u$ has constant sign, 
therefore $u(t)>0$ for all $t\in[t_0,t_1]$. 
This finishes the proof.
\end{proof}

This result extends Theorem \ref{teoCatFunIso} by taking $a=1$. 
We find it interesting because it could be used to construct Lyapunov functions of catenary type 
that in addition satisfies more properties as for example being a norm, if we are in a vector space.
For example, if we consider the differential equation $\dot x=-2x$ in $\R^n$, 
it is easy to see that for every norm $\liap \colon \R^n\to\R$ 
it holds that $\ddot\liap=4\liap$.

\subsection{Catenary functions for attractors}
\label{secCatAtt}
With the next Proposition we wish to remark that the Lyapunov functions obtained in 
Proposition \ref{propLiapCatAtt} are the catenary functions for an attractor.
In Theorem \ref{teoCatAttDiff} we construct a differentiable Lyapunov function  of catenary type
for an asymptotically stable equilibrium point in $\R^n$ of a $C^1$ differential equation. 
The result is based on Massera's Theorem on $C^1$ Lyapunov functions.

\begin{prop}
If $\Lambda$ is an attractor then every catenary function of $\Lambda$ satisfies $\dot\liap=-\liap$. 
For a repeller we obtain $\dot\liap=\liap$.
\end{prop}

\begin{proof}
  Given a catenary function $\liap$ for $\Lambda$ we know, by definition, that $\ddot\liap=\liap$. 
  Given a point $x$ in an isolating neighborhood of $\Lambda$ 
  we know that $\liap(\phi_t(x))=ae^t+be^{-t}$ for some 
  $a,b\in\R$ depending on $x$. 
  Since $\phi_t(x)$ is in the isolating neighborhood for all $t\geq 0$, 
  $\liap$ vanishes on $\Lambda$ and $\liap$ is continuous 
  we have that $a=0$. 
  Therefore, $\liap(\phi_t(x))=be^{-t}$ and $\dot\liap=-\liap$. This finishes the proof.
\end{proof}

\begin{thm}
\label{teoCatAttDiff}
Consider in $\R^n$ the differential equation $\dot x=f(x)$, 
$f\colon\R^n\to\R^n$ a function of class $C^1$, 
with an asymptotically stable equilibrium point $p\in\R^n$.  
Then, there is a differentiable positive Lyapunov function $\liap\colon U\to\R$, 
where $U$ is a neighborhood of $p$,  
satisfying $\dot\liap=-a\liap$ for some constant $a>0$. 
Moreover, $\liap$ is $C^1$ in $U\setminus\{p\}$.
\end{thm}

\begin{proof}
From \cite{Massera49}*{Theorem 8} we known that there is a $C^1$ Lyapunov function $V$ defined on a compact neighborhood $U'$ of $p$ 
that is positive in $U'\setminus\{p\}$, $V(p)=0$ and 
$\dot V<0$ in $U'\setminus\{p\}$.
Let $m=\min_{x\in\partial U'} V(x)$ and define 
$\Sigma=\{x\in U':V(x)=m/2\}$. 
We know that $\Sigma$ is a compact cross section. 
Moreover, since $V$ is $C^1$ and $\dot V\neq 0$ at $\Sigma$ we can apply the implicit function theorem 
to conclude that $\Sigma$ is a codimension-one submanifold of class $C^1$. 
For a value of $a$ that will be determined, define $\liap(\phi_t(x))=e^{-at}$ for all $x\in\Sigma$ and $t\geq 0$ 
and $\liap(p)=0$. 
If $U=\clos(\phi_{\R^+}(\Sigma))$ we have that $\liap$ is continuous 
in $U$. 
Moreover, $\liap$ is $C^1$ in $U\setminus\{p\}$ because $\Sigma$ and $\phi$ are $C^1$. 

We will show that there is a value of $a$ that makes $\liap$ differentiable at $p$. 
Consider the Euclidean norm $\|\cdot\|$ in $\R^n$, assume that $p$ is the origin and denote $\phi_t(x_0)$ as $x(t)$.
Since $f$ is $C^1$ there is $k>0$ such that $\|f(x)\|\leq k\|x\|$ for all $x\in U$. 
Then $\|\dot x\|\leq k\|x\|$ and 
 $\frac d{dt} \|x(t)\|^2 \geq -2\|x(t)\| \|\dot x(t)\|\geq -2k\|x(t)\|^2.$
If $u(t)=\|x(t)\|^2$ then $\dot u/u\geq -2k$. 
Thus, integrating we obtain $u(t)\geq e^{-2kt}$ and 
\begin{equation}
 \label{ecu1}
\|x(t)\|\geq e^{-kt}\|x_0\|. 
\end{equation}
Take $a=2k$ and $l>0$ such that 
\begin{equation}
\label{ecu2} 
\|x_0\|^{-2}\leq l,\,\forall x_0\in\Sigma. 
\end{equation}
Given $x\in U\setminus\{p\}$ there are $x_0\in\Sigma$ 
and $t\geq 0$ such that $x=\phi_t(x_0)$. 
Since $x_0\in\Sigma$ we have that $\liap(x_0)=1$ and  $\liap(x)=e^{-at}=e^{-2kt}= 
(e^{-kt})^2\leq \|x\|^2/\|x_0\|^2$. The last inequality follows from (\ref{ecu1}). 
Using (\ref{ecu2}) we obtain $\liap(x)\leq l\|x\|^2$. 
Since $\liap\geq 0$, this implies that $\liap$ is differentiable at the origin $p$ 
and the proof ends.
\end{proof}

\subsection{Fake singularities}
\label{subsecFakeSing}
The purpose of this section is to give a model, a semi-conjugacy, for an arbitrary isolated set. 
We will consider a special type of isolated point.

\begin{df}
 An isolated point $p$ of a partial flow $\phi$ is a \emph{fake singularity} 
 if there are $x_s$ and $x_u$, different from $p$, 
 such that: 1) if $x\neq p$ and $\lim_{t\to +\infty}\phi_t(x)=p$ then 
 $x=\phi_t(x_s)$ for some $t\in\R$ and 2) 
 if $x\neq p$ and $\lim_{t\to -\infty}\phi_t(x)=p$ then 
 $x=\phi_t(x_u)$ for some $t\in\R$.
\end{df}

Let $(\Sigma,\dist)$ be a metric space with a point $x_0\in\Sigma$ having a compact neighborhood. 
Define $X=\R\times \Sigma$ and $p=(0,x_0)\in X$. 
We say that a partial flow on $X$ has \emph{horizontal trajectories} 
if every trajectory is contained in a set of the form $\R\times\{x\}$.
Consider the projection $\pi_1(s,x)=s$.

\begin{prop}
\label{propFakeConst}
Consider a continuous function 
$W\colon X\to \R$ such that 
$W(p)=0$, $W(q)>0$ for all $q\neq p$ and 
\begin{equation}
\label{ecuFake}\int_{-1}^0\frac1{W(s,x_0)}ds=\int_0^1\frac1{W(s,x_0)}ds=\infty. 
\end{equation}
Then there is a unique partial flow in $X$ with horizontal trajectories, 
a fake singularity at $p$ and $\dot\pi_1=W$. 
% if $\liap\colon X\to \R$ is defined as $\liap(s,x)=s$ then $\dot \liap= W$.
\end{prop}

\begin{proof}
Given $(s,x)\in X$ 
let $u(t)$ be the solution of $\dot u(t)=W(u(t),x)$ such that $u(0)=s$. 
Define a partial flow $\phi$ on $X$ 
by 
$
  \phi_t(s,x)=(u(t),x).
$ 
We have a singular point at $p$. 
If we define $x_s=(-1,x_0)$ and $x_u=(1,x_0)$ we have that 
$p$ is a fake singularity;
and if we take a compact neighborhood $U\subset \Sigma$ of $x_0$ 
we have an isolating neighborhood $N=[-1,1]\times U$. 
See Figure \ref{figEntornoAdaptado}. 
Finally, we have 
$$
  \dot \pi_1(s,x)=\frac d{dt} \pi_1(\phi_t(s,x))|_{t=0}
  =\frac d{dt} \pi_1(u(t),x)|_{t=0}
  =\frac d{dt} (u(t)|_{t=0}
  =\dot u
  =W(s,x).
$$
This finishes the proof.
\end{proof}
\begin{figure}[ht]
\center 
\includegraphics{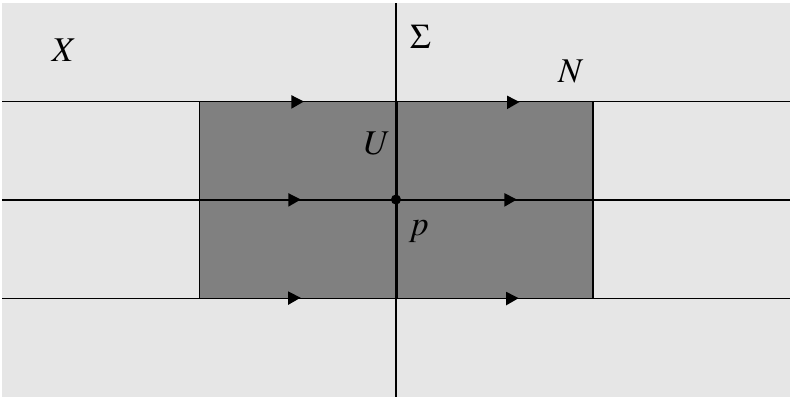}  
\caption{Singular flow box $N$ around the fake singularity $p$.}
\label{figEntornoAdaptado}
\end{figure}

\begin{rmk}
 A function $W$ satisfying the assumptions of the proposition is for example $W(s,x)=|s|+\dist(x,x_0)$, 
 where $\dist$ is the metric of $\Sigma$.
\end{rmk}

The partial flow given by Proposition \ref{propFakeConst}, or a conjugate one, 
will be called as $(W,\Sigma)$-\emph{fake singularity}.

\begin{prop}
 Every fake singularity is a $(W,\Sigma)$-fake singularity. 
\end{prop}

\begin{proof}
 Consider from Theorem \ref{teoCatFunIso}
 a catenary function 
 $\liap\colon N\to\R$, where $N$ is an isolating neighborhood of the fake singularity $p$. 
 For $\delta>0$ small define $\block =\{x\in N:\liap(x)\leq\delta\}$.
 Define $\Sigma=\{x\in \block:\dot\liap(x)=0\}$ with $p\in\Sigma$. 
 Define $\pi\colon \block\to\Sigma$ as 
 \[
  \pi(x)=\left\{
  \begin{array}{ll}
   \phi_t(x) & \hbox{ if }\phi_t(x)\in\Sigma \hbox{ and } \phi_{[0,t]}(x)\in N, \\
   p & \hbox{ in other case.}
  \end{array}
  \right.
 \]
 Define $h\colon \block\to\R\times\Sigma$ by 
 $h(x)=(\dot\liap(x),\pi(x))$. 
 Denote by $\psi$ the partial flow in $h(\block)$ 
 induced $\phi$ and $h$. 
It only rests to note that $\dot\pi_1=\ddot\liap=\liap$ (time derivative with 
respect to $\psi$). Then $\psi$ is a $(\liap,\Sigma)$-fake singularity and $h$ is a conjugacy with $\phi$.
\end{proof}

A product neighborhood $N$ as in Figure \ref{figEntornoAdaptado} is a \emph{singular flow box}.
Therefore, the previous result implies that every fake singularity is contained in a singular flow box. 

\begin{thm}
\label{teoSemiConj}
Every isolated set $\Lambda$ for a flow $\phi$ 
has an isolating neighborhood $N$ such that $\phi|N$ is semi-conjugate with a 
singular flow box around a fake singularity.
\end{thm}

\begin{proof}
Let $\liap\colon \block\to \R$ be a catenary function.
Define the sets 
\[
\begin{array}{c}
W^s=\{x\in \block:\phi_{\R^+}(x)\subset U\},\\
W^u=\{x\in \block:\phi_{\R^-}(x)\subset U\}.
\end{array}  
\]
Define an equivalence relation $\sim$ on $N$ generated by: 
$x\sim y$ if $x,y\in W^s\cup W^u$ and $\dot\liap(x)=\dot\liap(y)$. 
In the quotient, $\Lambda$ is a fake singularity and the projection is a semi-conjugacy.
\end{proof}

Figure \ref{figSemiConj} illustrates this result for a hyperbolic equilibrium point of saddle type in $\R^2$.

\begin{figure}[ht]
\center 
\includegraphics{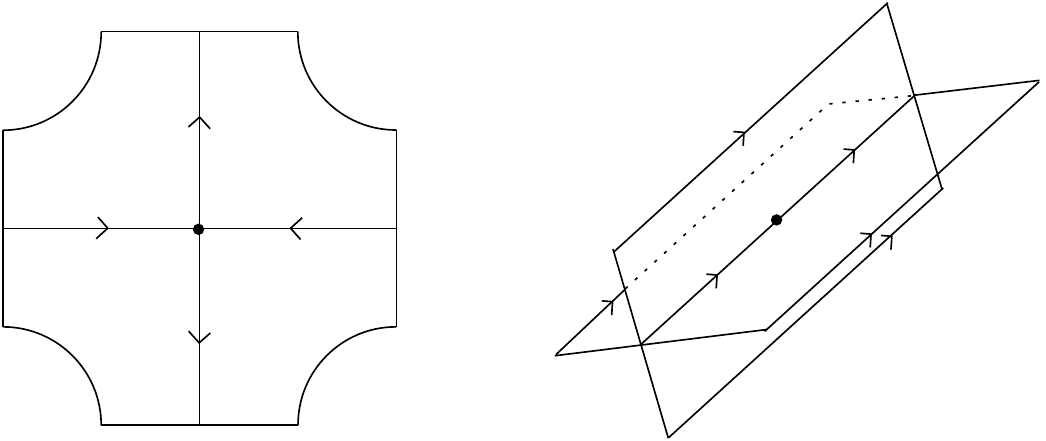}  
\caption{Left: a hyperbolic singular point. Right: its associated fake singularity by Theorem \ref{teoSemiConj}}
\label{figSemiConj}
\end{figure}

\section{Expansive flows} 
\label{secExpFlow}
The purpose of this section is to construct catenary functions and metrics for 
expansive flows on compact metric spaces. The first problem is to find an isolated set 
associated with the expansivity of a flow. 
Let $\phi\colon \R\times X\to X$ be a continuous flow on a compact metric space $(X,\dist)$. 

\begin{rmk}
In the case of expansive homeomorphisms we have that the diagonal is an isolated, but for flows this is not the case.
The diagonal $\{(x,x):x\in X\}$ is an isolated 
for the flow $\psi_t(x,y)=(\phi_t(x),\phi_t(y))$
if and only if $X$ is a finite set.
It is true, essentially, because 
$\psi_t(x,\phi_s(x))$ is close to the diagonal for all $t\in\R$ if $s$ is small.
\end{rmk}

According to \cite{BW} we say that $\phi$ is an \emph{expansive flow} 
if for all $\epsilon>0$ there is $\delta>0$ such that if $\dist(\phi_{h(t)}(x),\phi_t(y))<\delta$ for all 
$t\in\R$ with $h\colon\R\to\R$ a continuous function, $h(0)=0$, 
then there is $s\in (-\epsilon,\epsilon)$ such that 
$y=\phi_s(x)$. In order to find an isolated set, allowing us to apply our previous results, 
we will consider local cross sections.

\subsection{Local cross sections}
We start stating expansivity using cross sections.
In \cite{Whitney33} a local cross section is defined for each point. 
Since we need a family of such cross sections we sketch Whitney's construction.

Assume that the flow is regular, that is, it has no singular points. 
In this case it is easy to prove that 
for a flow $\phi$ on a compact metric space $X$
there are three positive parameters $\delta,\tau,a\in\R$ such that 
if $\dist(x,y)\leq\delta$ then 
$$\dist(x,\phi_\tau(y))-\dist(y,x)\geq a>0.$$ 
For $x,y\in X$ define $$\theta_x(y)=\int_0^\tau \dist(x,\phi_t(y))dt.$$
We have that $$\dot\theta_x(y)=\lim_{t\to 0}[\theta_x(\phi_t(y))-\theta_x(y)]t^{-1}=\dist(x,\phi_\tau(y))-\dist(x,y).$$
Therefore, if $\dist(x,y)\leq\delta$ then $\dot\theta_x(y)\geq a>0$. 
This implies that there is $\epsilon_0>0$ such that 
for all $\epsilon\in(0,\epsilon_0)$ the set 
$$H_\epsilon(x)=\{y\in X:\dist(x,y)\leq\epsilon,\theta_x(y)=\theta_x(x)\}$$
is a local cross section for each $x\in X$. 

\begin{rmk}
 We can consider the function $H_\epsilon\colon X\to\K(X)$ for $\epsilon$ fixed. 
 In this way we have that $H_\epsilon$ is semi-continuous, that is, 
 if $y_n\in H_\epsilon(x_n)$ with $x_n\to x$ and $y_n\to y$ then $y\in H_\epsilon(x)$. 
\end{rmk}

\begin{rmk}
 We can state expansivity as follows. 
 A flow $\phi$ is expansive if and only if there is $\epsilon>0$ such that 
 if for some continuous $h\colon\R\to\R$ and $x,y\in X$ we have that 
 $\phi_{h(t)}(y)\in H_\epsilon(\phi_t(x))$ for all $t\in\R$ 
 then $x=\phi_{h(0)}(y)$.
\end{rmk}

Define $$N_\epsilon=\{(x,y):y\in H_\epsilon(x)\}.$$ 
In $N_\epsilon$ we define a partial flow $\psi$ by 
\begin{equation}
  \label{ecuFlujoProd}
  \psi_t(x,y)=(\phi_t(x),\phi_{h(t)}(y))
\end{equation}
if $h\colon[0,t]\to\R$ is a continuous function such that 
$h(0)=0$ and $\phi_{h(s)}(y)\in H_\epsilon(\phi_s(x))$ for all $s\in[0,t]$. 
Notice that this function $h$ is unique if it exists. 
If $\phi$ is expansive then the trajectories of $\psi$ are not defined for all $t\in\R$. 
In fact, $\psi_t(x,y)$ is defined for all $t\in\R$ if and only if $y=x$, assuming expansivity. 
This is another equivalent way to state the expansivity of $\phi$. 
In other words: 
\begin{rmk}
\label{rmkIsoExpFlow}
The flow $\phi$ is expansive if and only if $\Lambda=\{(x,x):x\in X\}$ 
is an isolated set for the partial flow $\psi$ on $N_\epsilon$.   
\end{rmk}

Having found our isolated set associated to an expansive flow we are ready to apply our previous results.

\subsection{Catenary functions for expansive flows}
Given a flow $\phi$ on the compact metric space $X$ we consider the partial flow $\psi$ defined 
in (\ref{ecuFlujoProd}).
In this section the dots, indicating time derivatives, 
are related with $\psi$.

\begin{thm}
 If $\phi$ is an expansive flow without singular points on a compact metric space $X$ 
 then there is a continuous function 
 $\liap\colon N_\epsilon\to\R$ such that 
 $\liap(x,y)\geq 0$ with equality if and only if $x=y$
 and $\ddot \liap=\liap$.
\end{thm}

\begin{proof}
  It is a direct consequence of Theorem \ref{teoCatFunIso} and 
  Remark \ref{rmkIsoExpFlow}.
\end{proof}

\begin{df}
 A \emph{sectional metric} is a continuous function $$D\colon\{(x,y,z)\in X^3:y,z\in H_\epsilon(x)\}\to\R,$$
 that will be denoted as $D_x(y,z)=D(x,y,z)$,
 such that $D_x\colon H_\epsilon(x)\times H_\epsilon(x)\to\R$ is a metric 
 for each $x\in X$.
\end{df}

\begin{thm}
\label{teoCatMetExpFLow}
 Every expansive flow without singularities on a compact metric space $X$ admits 
 a sectional metric $D$ satisfying $\ddot D_x=D_x$ for all $x\in X$.
\end{thm}

\begin{proof}
 We know that the diagonal $\Lambda$ is an isolated set for $\psi$. 
 Then, there is a catenary pseudo-metric 
 $\di\colon N_\epsilon\times N_\epsilon\to\R$ given by Theorem \ref{teoCatMetSet}. 
 Define $$D_x(y,z)=\di((x,y),(x,z)).$$ 
 Since $\ddot \di=\di$ we have that $\ddot D_x=D_x$ for all $x\in X$. 
 We know that $\di((x_1,y_1),(x_2,y_2))=0$ if and only if 
 $(x_1,y_1)=(x_2,y_2)$ or $(x_1,y_1),(x_2,y_2)\in\Lambda$ (recall Definition \ref{dfCatPsMet}). 
 Then, $D_x(y,z)=0$ implies $y=z$. 
 We have that $D_x(y,z)=D_x(z,y)$, $D_x(y,y)=0$ and $D_x(y,z)\geq 0$ by the corresponding properties of 
 $\di$. 
 Let us prove the triangular inequality: 
 \[
   D_x(a,b)+D_x(b,c)=
   \di((x,a),(x,b))+\di((x,b),(x,c))
   \geq \di((x,a),(x,c))
   =D_x(a,c).
 \]
Finally, the continuity of $D$ follows by the continuity of $\di$.
\end{proof}

\section{Discrete dynamical systems}
\label{secLyapHomeo}

In this section we extend our results to the dynamics of homeomorphisms.
Applications to expansive and cw-expansive homeomorphisms are given. 

\subsection{Isolated sets for homeomorphisms}
Let $f\colon X\to X$ be a homeomorphism of a metric space $(X,\dist)$. 
An $f$ invariant set $\Lambda$ is \emph{isolated} if there 
is a compact neighborhood $N$ of $\Lambda$ such that
$f^n(x)\in N$ for all $n\in\Z$ implies that $x\in \Lambda$. 
% Fix an isolated set $\Lambda$ with isolating neighborhood $N$.

\subsection{Suspension flow} 
Fix $\nu>0$.
Consider $\susp X=\R\times X/\sim$ where 
$(s,x)\sim (t,y)$ if and only if 
$\nu(t-s)\in\Z$ and $y=f^{\nu(t-s)}(x)$.
% For example $(1,x)\sim (0,f(x))$ for all $x\in X$.
Denote by $\pi(s,x)$
the equivalence class of $(s,x)$. 
Let $\phi\colon \R\times \susp X\to \susp X$ 
be the \emph{suspension} flow defined by $\phi_t\pi(s,x)=\pi(s+t,x)$.

\begin{prop}
The set $\susp\Lambda=\pi(\R\times\Lambda)$ is isolated 
for $\phi$ if $\Lambda$ is an isolated set for $f$.
\end{prop}

\begin{proof}
If $N$ is an isolating neighborhood of $\Lambda$ define 
$\susp N=\pi([-3\nu/4,3\nu/4]\times N)$. 
Assume that $N$ is compact. 
Since $\pi$ is continuous we have that $\susp N$ is compact.
Let us show that $\susp N$ isolates $\susp\Lambda$.
Suppose that $\phi_t\pi(s,x)\in \susp N$ for all $t\in\R$. 
This means that $\pi(s+t,x)\in \susp N$ for all $t\in\R$. 
Define $t_n=n-s$ for each $n\in\Z$. 
Then $(n,x)$ is equivalent with some $(u_n,y_n)\in [-3\nu/4,3\nu/4]\times N$ 
for each $n\in\Z$. 
Since $n\in\Z$ and $n-u_n\in\Z$ we have that $u_n\in\Z\cap[-3\nu/4,3\nu/4]$. 
Thus $u_n=0$ for all $n\in\Z$ and $y_n=f^n(x)\in \Lambda$. 
This implies that $x\in N_f$ and $\pi(s,x)\in\susp\Lambda$. 
It only rests to note that $\susp N$ is a neighborhood of $\susp\Lambda$.
\end{proof}

\subsection{Catenary functions}
Given $\liap\colon N\to \R$ define 
\[
\begin{array}{ll}
  \dder \liap(x) & = [\liap(f(x))-\liap(x)] - [\liap(x)-\liap(f^{-1}(x))]\\
	     & = \liap(f(x))-2\liap(x)+\liap(f^{-1}(x))\\
\end{array}
\]
if $f(x),x,f^{-1}(x)\in N$.
\begin{df}[Discrete catenary]
A \emph{catenary function} for an isolated set $\Lambda$ 
is a continuous function $\liap\colon N\to\R$ 
such that $\dder \liap(x)=\liap(x)$, $\liap$ vanishes on $\Lambda$ and is positive in $N\setminus \Lambda$. 
\end{df}

\begin{thm}
\label{teoLyapDisc}
  Every isolated set of a homeomorphism admits a catenary function.
\end{thm}

\begin{proof}
Let $\lambda=\frac12(3+\sqrt 5)$, $T=\log\lambda$ and $\nu=T^{-1}$. 
Recall that $\nu$ is the parameter used to defined the suspension flow.
By Theorem \ref{teoCatFunIso} we have a catenary function $\susp\liap\colon \susp N\to\R$ 
for the suspension flow. 
For $x\in N$ define 
$\liap(x)=\susp\liap(\pi(0,x))$.
Since $\susp{\ddot\liap}=\susp\liap$, we have that for each 
$x\in N$ it holds that $\susp\liap(\phi_t(\pi(0,x)))=Ae^t+Be^{-t}$ for suitable constants 
$A,B\in\R$ depending on $x$. 
Notice that $\pi(0,f^{n}(x))=\phi_{nT}(\pi(0,x))$ for all $n\in\Z$.
Then 
\[
\begin{array}{ll}
\dder \liap(x) & = \liap(f(x))-2\liap(x)+\liap(f^{-1}(x))\\
	   & = \susp\liap(\pi(0,f(x)))-2\susp\liap(\pi(0,x))+\susp\liap(\pi(0,f^{-1}(x)))\\
	   & = \susp\liap(\phi_T(\pi(0,x)))-2\susp\liap(\pi(0,x))+\susp\liap(\phi_{-T}(\pi(0,x)))\\
	   & = [Ae^T+Be^{-T}]-2[A+B]+[Ae^{-T}+Be^{T}]\\
	   & = (A+B)(e^T-2+e^{-T})\\
	   & = (A+B)(\lambda-2+\lambda^{-1})\\
	   & = A+B = \susp\liap(\pi(0,x))=\liap(x).
\end{array}
\]
It only rests to note that $\liap$ is positive in $N\setminus\Lambda$, $\liap(\Lambda)=0$ and 
$\liap$ is continuous by construction.
\end{proof}

\subsection{Catenary metric for an isolated set of a homeomorphism}

\begin{df}
 A \emph{catenary pseudo-metric} for an isolated set $\Lambda$ is a pseudo-metric $\di\colon N\times N\to\R$, with $N$ 
 an isolating neighborhood of $\Lambda$,
 such that $\ddot\di=\di$ and 
 $\di(x,y)=0$ if and only if $x=y$ or $x,y\in\Lambda$.
\end{df}

\begin{thm}
\label{teoCatMet}
Every isolated set admits a catenary pseudo-metric.
\end{thm}

\begin{proof}
 It follows with the techniques of the proof of Theorem \ref{teoLyapDisc}, considering the suspension, and
 applying Theorem \ref{teoCatMetSet}.
\end{proof}

\subsection{Expansive homeomorphisms}
A homeomorphism $f\colon X\to X$ of a compact metric space $(X,\dist)$ is \emph{expansive} 
if there is $\delta>0$ such that 
if $\dist(f^n(x),f^n(y))\leq\delta$ for all $n\in \Z$ then $x=y$. 
In this section we will prove that every expansive homeomorphism admits a Lyapunov function of catenary type defined 
for small compact subsets of $X$.
We also construct a catenary local metric for an arbitrary expansive homeomorphism of a compact metric space. 
Assuming that this catenary local metric is locally minimizing we obtain a catenary metric for the expansive homeomorphism.

\subsubsection{Catenary Lyapunov functions for expansive homeomorphisms}

Define $$\K_\delta(X)=\{A\in \K(X):\diam(A)\leq\delta\}.$$

\begin{thm}
\label{catExp}  
For every expansive homeomorphism $f\colon X\to X$ of a compact metric space 
there are $\delta>0$ and a continuous function 
$\liap\colon \K_\delta(X)\to\R$ such that: 
\begin{enumerate}
  \item $\liap(A)\geq 0$ for all $A\in \K_\delta(X)$, with equality if and only if $A$ is a singleton,
  \item $\dder \liap=\liap$.
\end{enumerate}
\end{thm}

\begin{proof}
If we define $Y=\K(X)$ and 
$g\colon Y\to Y$ as $g(A)=\{f(x):x\in A\}$ we 
have that $f$ is expansive if and only if $\Lambda=\{\{x\}:x\in X\}$ 
is an isolated set of $g$. 
Take $\liap$ a catenary function from Theorem \ref{teoLyapDisc} defined in a neighborhood $\K_\delta(X)$ 
of $\Lambda$ for $\delta>0$ small.
\end{proof}

Define $X^2_\delta=\{(x,y)\in X\times X:\dist(x,y)\leq\delta\}$.
The following result extends the onw obtained by Lewowicz in \cite{Lew89}.

\begin{coro}
  For every expansive homeomorphism $f\colon X\to X$ of a compact metric space there are $\delta>0$ 
  and a non-negative continuous function $\liap\colon X^2_\delta\to\R$ such that 
  \begin{enumerate}
    \item $\liap(x,y)= 0$ if and only if $x=y$,
    \item $\ddot\liap(x,y)>0$ if $x\neq y$.
  \end{enumerate}
\end{coro}

\begin{proof}
 Restrict the catenary function of Theorem \ref{catExp}  
to pairs of points.
\end{proof}

\subsubsection{Catenary local metrics for an expansive homeomorphism}
Define $$X^3_\delta=\{(x,y,z)\in X\times X\times X: y,z\in B_\delta(x)\}.$$

\begin{df}
 A \emph{local metric} in $X$ is a continuous function 
 $D\colon X^3_\delta\to\R$, that will be denoted as $D(x,y,z)=D_x(y,z)$, satisfying: 
 \begin{enumerate}
  \item for each $x\in X$, $D_x\colon B_\delta(x)\times B_\delta(x)\to\R$ is a metric, 
  \item $D_x(x,y)=D_y(x,y)$ if $\dist(x,y)\leq\delta$.
 \end{enumerate}
\end{df}

\begin{df}
 A local metric $D$ is \emph{catenary} if 
 $\ddot D_x=D_x$ for all $x\in X$.
\end{df}

\begin{thm}
\label{teoExpCatLocMet}
 Every expansive homeomorphism of a compact metric space admits a catenary local metric.
\end{thm}

\begin{proof}
Define $F_2(X)=\{\{x,y\}:x,y\in X\}$.
Expansivity is equivalent with the space of singletons $\Lambda=\{\{x\}:x\in X\}$ 
 being an isolated set of the homeomorphism 
 $g\colon F_2(X)\to F_2(X)$ given by $g(\{x,y\})=\{f(x),f(y)\}$. 
 By Theorem \ref{teoCatMet} there is a catenary pseudo-metric $\di\colon N\times N\to \R$ 
 where $N\subset F_2(X)$ 
 is an isolating neighborhood of $\Lambda$.
 Take $\delta>0$ such that if $\dist(x,y)\leq \delta$ then $\{x,y\}\in N$ and
 define 
 $$D_x(y,z)=\di(\{x,y\},\{x,z\})$$ 
 if $y,z\in B_\delta(x)$.
 
 Let us prove that $D_x$ is a metric in $B_\delta(x)$. 
 It is easy to see that $D_x(y,z)\geq 0$, $D_x(y,z)=D_x(z,y)$ and $D_x(y,y)=0$ for all $y,z\in B_\delta(x)$. 
 Suppose that $D_x(y,z)=0$. Then $\di(\{x,y\},\{x,z\})=0$. 
 Since $\di$ is a pseudo-metric for $\Lambda$ we have that $y=z$. 
 Let us show the triangular inequality: 
 \[
 \begin{array}{ll}
  D_x(y,z)+D_x(z,u) & =\di(\{x,y\},\{x,z\})+\di(\{x,z\},\{x,u\})\\ 
  & \geq \di(\{x,y\},\{x,u\})\\
  &=D_x(y,u).  
 \end{array}
 \]
 
 We also have that $D_x(x,y)=\di(\{x,x\},\{x,y\})=\di(\{y,y\},\{x,y\})=D_y(x,y)$.
The continuity of $D$ follows by the continuity of $\di$. 
Finally, the catenary condition of $D$ follows by the corresponding property of $\di$. 
This finishes the proof.
\end{proof}

\subsubsection{Catenary metrics}
\label{secCatMetExp}
Here we consider the problem of constructing catenary metrics instead of pseudo-metrics.

\begin{df}
 A local metric $D$ is \emph{locally minimizing} if there is $\delta>0$ such that if 
 $\dist(x,y)<\delta$ then
 $D_x(x,y)\leq D_z(x,y)$ for all $z\in B_\delta(x)\cap B_\delta(y)$.
\end{df}

\begin{ejp}
 Let $M$ be a compact manifold with a Riemannian metric. 
 Denote by $\|\cdot\|$ the induced norm, consider the exponential map $\exp_x\colon T_xM\to M$ 
 and define $T^r_xM=\{v\in T_M:\|v\|\leq r\}$. 
 Take $r$ such that $\exp_x\colon T^r_xM\to M$ is a homeomorphism onto its image. 
 We can define a local metric by $D_x(y,z)=\|\exp_x^{-1}(y)-\exp_x^{-1}(z)\|$. 
 Moreover, it is locally minimizing.
\end{ejp}

Returning to our metric space $X$, given $x,y\in X$ define 
$$C_n^\delta(x,y)=\{a\in X^n:a_1=x,a_n=y,\dist(a_i,a_{i+1})\leq\delta\,\forall i=1,\dots,n-1\}$$
where $a_1,\dots,a_n$ are the coordinates of $a$, i.e., $a=(a_1,\dots,a_n)$.

\begin{rmk}
 Notice that for some $x,y$ it could be the case that $C_n^\delta(x,y)=\emptyset$ for all $n\geq 2$. 
 Also note that the relation $x\sim_\delta y$ if there is $n\geq 2$ such that $C_n^\delta(x,y)\neq \emptyset$, is an 
 equivalence relation on $X$. 
 This relation makes a partition of $X$ that allows us to separate the study into these equivalence classes. 
 For simplicity, we will assume that $X$ is connected and consequently for all $x,y$ and for all $\delta>0$ 
 there is $n\geq 2$ such that $C_n^\delta(x,y)\neq\emptyset$. 
\end{rmk}

\begin{df}
 A metric $\rho$ in $X$ is a \emph{catenary metric} if it is a metric defining the topology of $X$ and 
 there is $\delta>0$ such that $\ddot\rho(x,y)=\rho(x,y)$ whenever $\dist(x,y)\leq\delta$.
\end{df}

We do not know if every expansive homeomorphism of a compact metric space admits a catenary metric. 
The following is a partial result in this direction.

\begin{thm}
 If $D$ is a catenary and locally minimizing local metric on $X$ then 
 \[
    \rho(x,y)=\inf_{n\geq 2}\inf_{a\in C^\delta_n(x,y)}\sum_{i=1}^{n-1}D_{a_i}(a_i,a_{i+1})
 \]
is a catenary metric.
\end{thm}

\begin{proof}
It is easy to prove the triangular inequality for $\rho$, 
that $\rho(x,x)=0$ and $\rho(x,y)=\rho(y,x)\geq 0$ for all $x,y\in X$. 
Take $\delta>0$ such that if 
 $\dist(x,y)<\delta$ then
 $D_x(x,y)\leq D_z(x,y)$ for all $z\in B_\delta(x)\cap B_\delta(y)$.
This implies that $\rho(x,y)=D_x(x,y)$ if $\dist(x,y)<\delta$. 
Then, $x\neq y$ implies that $\rho(x,y)\neq 0$. 
Also, $\rho$ defines the topology of $X$. 
The catenary condition of $\rho$ follows by the corresponding property of $D$.
\end{proof}

\begin{ejp}[Shift map]
 Let $X=2^\Z$ be the space of sequences $\dots,x_{-1},x_0,x_1,x_2,\dots$ 
 such that $x_n\in\{0,1\}$ for each $n\in\Z$. 
 Consider the distance 
 \[
  \di(x,y)=\sum_{n\in\Z}\frac{|x_n-y_n|}{\lambda^{|n|}}
 \]
where $\lambda>1$, $x,y\in X$. Define $f\colon X\to X$, the shift map, 
by $(f(x))_n=x_{n+1}$. It is easy to see that this metric satisfies 
$\ddot \di(x,y)=a\di(x,y)$ where $a=\lambda+\lambda^{-1}-2$ and $x,y$ are close enough. 
For $\lambda=\frac 1 2(3+\sqrt 5)$ we have that $\di$ is a catenary metric.
\end{ejp}

\begin{ejp}[Pseudo-Anosov diffeomorphisms]
 Let $X$ be a compact surface without boundary. 
 It is known \cites{Lew89,Hi} that every expansive homeomorphism 
 $f\colon X\to X$ is conjugate with a pseudo-Anosov diffeomorphism. 
 By definition, pseudo-Anosov diffeomorphisms have two 
 transverse and invariant singular foliations $F_s,F_u$ with transverse measures $\mu_s,\mu_u$. 
 There is a parameter $\lambda>1$ such that 
 the unstable measure is expanded $\lambda$ by the diffeomorphism and 
 the stable measure is contracted by $\lambda$ by $f$. 
 These measures define naturally a metric $\di$ on $X$ 
 satisfying $\ddot \di(x,y)=a\di(x,y)$ where $a=\lambda+\lambda^{-1}-2$ and $x,y$ are close enough. 
\end{ejp}

\subsection{Cw-expansive homeomorphisms}
Let $(X,\dist)$ be a compact metric space.
Recall that a \emph{continuum} is a compact connected set.
Denote by $$\KC(X)=\{A\in\K(X):A\hbox{ is continuum} \},$$ the space of subcontinua of $X$, and for $\delta\geq 0$
define $$\KC_\delta(X)=\{A\in\KC(X):\diam(A)\leq\delta\}.$$
Following \cite{Kato} we say that a homeomorphism $f\colon X\to X$ is \emph{cw-expansive} if there is 
$\delta>0$ such that if 
$f^n(A)\in\KC_\delta(X)$ for all $n\in \Z$ then $A\in \KC_0(X)$. 
We have a result similar to Theorem \ref{catExp}.
We have to replace $\K(X)$ with $\KC(X)$ in the domain of the function $\liap$.

\begin{thm}
\label{catCwExp}  
For every cw-expansive homeomorphism $f\colon X\to X$ of a compact metric space 
there are $\delta>0$ and a continuous function 
$\liap\colon \KC_\delta(X)\to\R$ such that: 
\begin{enumerate}
  \item $\liap(A)\geq 0$ for all $A\in N$, with equality if and only if $A$ is a singleton,
  \item $\dder \liap=\liap$.
\end{enumerate}
\end{thm}

\begin{proof}
  Is similar to the proof of Theorem \ref{catExp}.
\end{proof}

\subsection{Other forms of expansivity} 
In addition to cw-expansivity there are many other variation of the concept introduced by Utz. 
Let us mention some of them and indicate if an isolated set can be found in order to apply our results.
Given $N>0$, a homeomorphism $f\colon X\to X$ is $N$-\emph{expansive} (Morales \cite{Mo}) if there is $\delta>0$ 
such that if $\diam(f^n(A))<\delta$ for all $n\in\Z$ and some subset $A\subset X$ 
then $A$ has at most $N$ points. 
Notice that $1$-expansivity is expansivity. 
It is natural to consider the space $F_n(X)=\{A\in \K(X):|A|\leq n\}$ 
where $|A|$ denotes the cardinality of $A$. 
Note that $F_n$ are invariant compact subsets of $\K(X)$. 
Define $Y=(F_N(X)\setminus F_{N-1}(X))\cup F_1(X)$. It is invariant by $f$.
We have that $N$-expansivity is equivalent with: 
there is an open set $U\subset Y$ 
such that $F_1(X)\subset U$ 
and if $f^n(A)\in U$ for all $n\in\Z$ then $A\in F_1(X)$. 
It looks like the definition of isolated set but such an open set $U$ 
cannot have compact closure unless $X$ is a finite set. 
Therefore, we are not able to apply our result to $N$-expansivity.

We consider a definition given in \cite{Ar}.
For $\delta\geq 0$, a set $A\subset X$ is $\delta$-\emph{separated} if for all $x\neq y$, 
$x,y\in A$, it holds that $\dist(x,y)>\delta$. 
The $\delta$-\emph{cardinality} of a set $A$ is 
\[
 |A|_\delta=\sup\{|B|:B\subset A\hbox{ and }B\hbox{ is }\delta\hbox{-separated}\}.
\]
Given integer numbers $m>n\geq 1$ we say that $f\colon X\to X$ is 
$(m,n)$-\emph{expansive} if there is $\delta>0$ such that 
if $|A|=m$ then there is $k\in\Z$ such that $|f^ k(A)|_\delta>n$. 
For the special case $m=n-1$ we have that $(n-1,n)$-expansivity is equivalent with the set $F_{n-1}(X)$
being isolated in $F_n(X)$. For $m<n-1$ the situation is similar to $N$-expansivity explained above.

We were not able to find a connection between 
isolated sets and
point-wise expansivity
\cite{Reddy70}
or h-expansivity
\cite{Bowen72}.

% \end{preview}

\begin{bibdiv}
\begin{biblist}

\bib{Ab}{article}{
author={F. Abadie},
title={Enveloping actions and Takai duality for partial actions},
journal={Journal of Functional Analysis},
volume={197},
year={2003},
pages={14--67}}

\bib{Ar}{article}{
author={A. Artigue},
title={Finite sets with fake observable cardinality}, 
journal={Bulletin of the Korean Mathematical Society}, 
year={2014}}

\bib{AS}{article}{
author={J. Auslander},
author={P. Seibert},
title={Prolongations and stability in dynamical systems},
journal={Annales de l'Inst. Fourier}, 
volume={14},
year={1964}, 
pages={237--267}}

\bib{BaSz}{book}{
author={N. P. Bhatia},
author={G. P. Szeg\"o},
title={Dynamical Systems: Stability Theory and Applications},
publisher={Springer-Verlag},
series={Lect. Not. in Math.},
volume={35},
year={1967}}

\bib{Bowen72}{article}{
author={R. Bowen},
title={Entropy-expansive maps},
journal={Trans. of the AMS},
volume={164},
year={1972},
pages={323--331}}


\bib{BW}{article}{
author={R. Bowen and P. Walters}, 
title={Expansive one-parameter flows}, 
journal={J. Diff. Eq.}, year={1972}, pages={180--193},
volume={12}}

\bib{CE}{article}{
journal={Trans. of the Amer. Math. Soc.},
volume={158}, 
year={1971},
title={Isolated Invariant Sets and Isolated Blocks},
author={C. Conley},
author={R. Easton}}

\bib{Conley78}{book}{
author={C. Conley},
title={Isolated invariant sets and the Morse index},
series={Conference Board of the Mathematical Sciences}, 
volume={38},
year={1978},
publisher={AMS}}

\bib{Conley88}{article}{
author={C. Conley},
year={1988},
title={The gradient structure of a flow: I},
journal={Ergodic Theory and Dynamical Systems}, 
volume={8}, 
pages={11--26}}

\bib{Ch}{article}{
author={R. C. Churchill},
title={Isolated Invariant Sets in Compact Metric Spaces},
journal={J. Diff. Eq.},
volume={12}, 
pages={330--352},
year={1972}}

\bib{DH}{article}{
author={J. Denzler},
author={A. M. Hinz},
title={Catenaria Vera - The True Catenary},
journal={Expo. Math.},
year={1999},
volume={17},
pages={117--142}}

\bib{Fa}{article}{
author={A. Fathi},
title={Expansiveness, hyperbolicity and Hausdorff dimension},
journal={Commun. Math. Phys.},
volume={126},
year={1989},
pages={249--262}}

\bib{Hi}{article}{
author={K. Hiraide},
title={Expansive homeomorphisms of compact surfaces are pseudo-Anosov},
journal={Osaka J. Math.}, 
volume={27},
year={1990}, 
pages={117--162}}

\bib{HY}{book}{
author={J. Hocking},
author={G. Young},
title={Topology},
publisher={Addison-Wesley Publishing Company, Inc.},
year={1961}}

\bib{Hur}{article}{
author={M. Hurley},
title={Lyapunov Functions and Attractors in Arbitrary Metric Spaces},
year={1998},
journal={Proc. of the Am. Math. Soc.},
volume={126},
pages={245--256}}

\bib{IN}{book}{
author={A. Illanes},
author={S. B. Nadler Jr.},
title={Hyperspaces: Fundamentals and Recent Advances},
year={1999},
publisher={Marcel Dekker, Inc.}}

\bib{Kato}{article}{
author={H. Kato},
title={Continuum-wise expansive homeomorphisms},
journal={Can. J. Math.},
volume={45},
number={3},
year={1993},
pages={576--598}}

\bib{Kra}{book}{
author={N. N. Krasovskii},
title={Stability of Motion},
publisher={Stanford University Press},
year={1963}}

\bib{Lew80}{article}{
author={J. Lewowicz},
year={1980},
title={Lyapunov Functions and Topological Stability},
journal={J. Diff. Eq.},
volume={38},
pages={192--209}}

\bib{Lew89}{article}{
author={J. Lewowicz},
title={Expansive homeomorphisms of surfaces},
journal={Bol. Soc. Bras. Mat.}, 
volume={20}, 
pages={113--133}, 
year={1989}}

\bib{Massera49}{article}{
author={J. L. Massera},
title={On Liapunoff's Conditions of Stability},
journal={Ann. of Math.},
volume={50},
number={3},
pages={705--721},
year={1949}}

\bib{Mo}{article}{
author={C. A. Morales},
title={A generalization of expansivity},
journal={Discrete Contin. Dyn. Syst.},
volume={32},
year={2012}, 
number={1},
pages={293--301}}

% \bib{Nadler}{book}{
% author={S. Nadler Jr.},
% title={Hyperspaces of Sets},
% publisher={Marcel Dekker Inc. New York and Basel},
% year={1978}}

\bib{Pa}{article}{
author={M. Paternain},
title={Expansive flows and the fundamental group},
journal={Bull. Braz. Math. Soc.},
number={2},
volume={24},
pages={179--199},
year={1993}}

\bib{Reddy70}{article}{
author={W. L. Reddy},
title={Pointwise expansion homeomorphisms},
journal={J. Lond. Math. Soc.},
year={1970},
volume={2},
pages={232--236}} 

\bib{Ur}{article}{
author={R. Ures},
title={On expansive covering maps},
journal={Publicaciones Matemáticas del Uruguay},
volume={3},
year={1990}, 
pages={59--67}}

\bib{Utz}{article}{
author={W. R. Utz},
title={Unstable homeomorphisms},
journal={Proc. Amer. Math. Soc.},
year={1950},
volume={1},
number={6},
pages={769--774}}

\bib{Vi}{article}{
author={J. L. Vieitez},
title={Lyapunov functions and expansive diffeomorphisms on 3D-manifolds},
journal={Ergod. Theor. Dyn. Syst.}, 
volume={22},
year={2002},
pages={601-632}}

\bib{Whitney33}{article}{
author={H. Whitney},
title={Regular families of curves},
journal={Ann. of Math.}, 
number={34},
year={1933}, 
pages={244--270}}

\bib{WY}{article}{
author={F. W. Wilson Jr.},
author={J. A. Yorke},
title={Lyapunov functions and isolating blocks},
journal={J. Diff. Eq.},
volume={13},
year={1973},
pages={106--123}}
\end{biblist}
\end{bibdiv}
\end{document}